\documentclass[12pt]{amsart}

\usepackage[utf8]{inputenc}
\usepackage{amsmath,amsthm,amssymb}
\usepackage{indentfirst}
\usepackage{url}
\usepackage{tikz-cd}
\usepackage[colorlinks=true]{hyperref}
\usepackage{colonequals}
\usepackage{graphicx}
\usepackage{enumitem}
\usepackage{soul}
\usepackage{stmaryrd}
\usepackage{mathrsfs,mathtools}
\usepackage{tgpagella}
\usepackage{bbm}
\usepackage{subcaption}
\usepackage[alphabetic]{amsrefs}
\usepackage{xcolor}

\usepackage[margin=1in]{geometry}

\newtheorem{thm}[subsection]{Theorem}
\newtheorem{lem}[subsection]{Lemma}
\newtheorem{conj}[subsection]{Conjecture}
\newtheorem{prop}[subsection]{Proposition}

\newtheorem{cor}[subsection]{Corollary}
{
\theoremstyle{definition}
\newtheorem{remx}[subsection]{Remark}

\newtheorem{defnx}[subsection]{Definition}
\newtheorem{examplex}[subsection]{Example}
}
\definecolor{gold}{RGB}{255,215,0}
\newenvironment{rem}
{\pushQED{\qed}\remx}
{\popQED\endremx}

\newenvironment{defn}
{\pushQED{\qed}\defnx}
{\popQED\enddefnx}
\newenvironment{example}
{\pushQED{\qed}\examplex}
{\popQED\endexamplex}

\newcommand{\ZZ}{\mathbb Z}
\newcommand{\RR}{\mathbb R}
\newcommand{\CC}{\mathbb C}

\newcommand{\supp}{\mathrm{supp}}
\newcommand{\Newton}{\mathrm{Newton}}
\newcommand{\wt}{\mathrm{wt}}
\newcommand{\BD}{\mathcal{BD}}
\newcommand{\SBD}{\mathcal{SBD}}
\renewcommand{\top}{\mathrm{top}}

\newcommand{\BPD}{\mathrm{BPD}}
\newcommand{\RBPD}{\mathrm{RBPD}}
\newcommand{\MBPD}{\mathrm{MBPD}}

\newcommand{\SM}{\mathrm{SM}}

\newcommand{\GL}{\mathrm{GL}}

\newcommand{\word}{\mathrm{word}}

\begin{document}
\title{M-convexity of Grothendieck polynomials via bubbling}
\author{Elena S.~Hafner}

\address{Elena S.~Hafner, Department of Mathematics, Cornell University, Ithaca, NY 14853. \newline\textup{esh83@cornell.edu}
}

\author{Karola M\'esz\'aros}

\address{Karola M\'esz\'aros, Department of Mathematics, Cornell University, Ithaca, NY 14853. \newline\textup{karola@math.cornell.edu}
}

\author{Linus Setiabrata}
\address{Linus Setiabrata, Department of Mathematics, University of Chicago, Chicago, IL, 60637. \newline\textup{linus@math.uchicago.edu}
}

\author{Avery St.~Dizier}
\address{Avery St.~Dizier, Department of Mathematics, University of Illinois at Urbana-Champaign, Urbana, IL 61801. \newline\textup{stdizie2@illinois.edu}
}
\thanks{Elena S.~Hafner, Karola M\'esz\'aros  and Avery St.~Dizier received support from CAREER NSF Grant DMS-1847284. Avery St.~Dizier also  received support from NSF Grant DMS-2002079.}

\begin{abstract}
We introduce bubbling diagrams and show that they compute the support of the Grothendieck polynomial of any vexillary permutation. Using these diagrams, we show that the support of the top homogeneous component of such a Grothendieck polynomial coincides with the support of the dual character of an explicit flagged Weyl module. We also show that the homogenized Grothendieck polynomial of a vexillary permutation has M-convex support. 
\end{abstract}

\maketitle

\section{Introduction}

Grothendieck polynomials $\mathfrak G_w$ are multivariate polynomials associated to permutations $w \in S_n$. Grothendieck polynomials were introduced by Lascoux and Sch\"utzenberger \cite{ls1982} as representatives of the classes of Schubert varieties in the $K$-theory of the flag manifold. They generalize Schubert polynomials, which in turn generalize the classical Schur polynomials, a well-known basis of the ring of symmetric functions.

There has been a flurry of research on the support of Grothendieck polynomials as well as the distribution of their coefficients within their support \cites{hmms2022,weigandt2021,psw2021,hafner2022,mssd2022,ps2022,ccmm2022}. With Huh and Matherne, the second and fourth author conjectured that homogenized Grothendieck polynomials are Lorentzian (up to appropriate normalization). In particular, this conjecture would imply that their support is \emph{M-convex}, equivalently the set of integer points in a generalized permutahedron. That the support is the set of integer points of a convex polytope was previously conjectured by Monical, Tokcan, and Yong in \cite{mty2019}.

To date, it is known that homogenized Grothendieck polynomials are M-convex for several families of permutations. These include permutations of the form $1\pi$ with $\pi$ dominant on $\{2,3 \ldots, n\}$ (\cite{ms2020}), Grassmannian permutations (\cite{ey2017}), and permutations whose Schubert polynomial has all nonzero coefficients equal to $1$ (\cite{ccmm2022}). In the present paper, we prove M-convexity for homogenized Grothendieck polynomials of vexillary permutations. 

Our inspiration is the proof of the analogous result for all Schubert polynomials \cite{fms2018}. The latter relies heavily on the theory of dual Weyl characters which has no K-theoretic counterpart. Mimicking the dual Weyl character approach, we introduce \emph{bubbling diagrams}, which are diagrams (subsets of the $n\times k$ grid) endowed with additional data affecting the legality of certain local transformations (Definitions~\ref{defn:bubbling-move} and~\ref{defn:K-bubbling-move}). These diagrams also bear strong similarities with \emph{ghost diagrams} \cite{ry2015}. We show that bubbling diagrams compute the support of any vexillary Grothendieck polynomial:
\begin{thm}
	\label{thm:main1}
	If $w \in S_n$ is a vexillary permutation, then $\supp(\mathfrak G_w) = \{\wt(\mathcal D)\colon \mathcal D \in \BD(w)\}$.
\end{thm}
We also provide a much simpler subset $\SBD(w)\subseteq \BD(w)$ in Definition \ref{defn:sbd} which still realizes the conclusion of Theorem \ref{thm:main1}.
 
From the characterization of $\supp(\mathfrak G_w)$ afforded by Theorem~\ref{thm:main1}, we derive two interesting consequences. For a diagram $D$, let $\chi_D$ denote the dual character of the flagged Weyl module of $D$ (See Section~\ref{sec:background} for definitions). Denote the top degree component of $\mathfrak G_w$ by $\mathfrak G_w^\top$.

\begin{thm}
\label{thm:main2}
Let $w \in S_n$ be a vexillary permutation. There is a diagram $D^\top(w)$ such that $\supp(\mathfrak G_w^\top) = \supp(\chi_{D^\top(w)})$.
\end{thm}

For vexillary permutations, Theorem~\ref{thm:main2} implies that the \emph{Rajchgot polynomials} of \cite{psw2021} are dual characters of flagged Weyl modules. Consequently, their Newton polytopes are \emph{Schubitopes}, a subclass of generalized permutahedra introduced in \cite{mty2019} whose defining inequalities are derived from a diagram.

In a recent work, Pan and Yu \cite{PanYu23} also construct a diagram whose weight is the leading monomial of $\mathfrak G_w^\top$.  In general, their diagram is distinct from $D^\top(w)$; in particular, the dual character of their diagram does not have the same support as $\mathfrak G_w^\top$ for vexillary permutations. Tianyi Yu communicated to us that   the results of \cite{PanYu23} along with those in  \cite{Yu23} can be used to show that $\supp(\mathfrak G_w^\top)$ is the set of integer points of a Schubitope.

\begin{thm}
\label{thm:main3}
Let $w\in S_n$ be a vexillary permutation. Then the homogenized Grothendieck polynomial $\widetilde{\mathfrak G}_w$ has M-convex support. In particular, each homogeneous component of $\mathfrak G_w$ has M-convex support. 
\end{thm}

The lowest-degree homogeneous component of $\mathfrak{G}_w$, the Schubert polynomial, equals an integer multiple of some $\chi_{D}$ for any permutation.  As a consequence of Theorem \ref{thm:main2}, one might wonder whether this is the case for $\mathfrak G_w^\top$ or for other homogeneous components of $\mathfrak G_w$.

We can use the following result to verify whether or not the Newton polytopes of the homogeneous components of $\mathfrak G_w$ are Schubitopes, the Newton polytopes of the polynomials $\chi_D$.

\begin{thm}
	\label{thm:main4}
	Fix $n\geq 1$. The rank functions $r_{\SM_n(I)}$ of Schubert matroids form a basis of the vector space of functions $f\colon 2^{[n]} \to \RR$ satisfying $f(\emptyset) = 0$. In particular,
	\begin{itemize}
		\item A generalized permutahedron is a Schubitope if and only if its associated submodular function is a $\ZZ_{\geq 0}$-linear combination of rank functions of Schubert matroids, and
		
		\item Two Schubitopes $\mathcal S_D$ and $\mathcal S_{D'}$ are equal if and only if $D$ can be obtained from $D'$ by a permutation of columns.
	\end{itemize}
\end{thm}

Using Theorem \ref{thm:main4}, we exhibit two interesting counter examples:

(1) Example~\ref{exp:main3counter} provides a vexillary permutation $w$ and a (not top-degree) homogeneous component of $\mathfrak G_w$ whose Newton polytope is not a Schubitope. This suggests focusing attention only on $\mathfrak G_w^\top$ when looking for Schubitopes among the homogeneous components of $\mathfrak G_w$.

(2) Example~\ref{exp:main2counter} gives a non-vexillary permutation $w$ where $\Newton(\mathfrak G_w^\top)$ is not a Schubitope and so is not a multiple of any $\chi_D$. This suggests restricting attention to vexillary permutations when relating $\mathfrak G_w^\top$ to $\chi_D$. We conjecture the following strengthening of Theorem \ref{thm:main2} (tested for all vexillary $w\in S_n$, $n\leq 9$):

\begin{conj}
	\label{conj:1}
	If $w \in S_n$ is vexillary, then $\mathfrak G_w^\top$ is an integer multiple of $\chi_{D^\top(w)}$.
\end{conj}

\textbf{Outline of the paper.} In Section~\ref{sec:background}, we recall some relevant background. In Section~\ref{sec:bubbling-and-supports}, we establish basic properties of bubbling diagrams, including a non-recursive characterization of the set of bubbling diagrams (Lemma~\ref{lem:BD-characterization}) and prove Theorem~\ref{thm:main1} by constructing weight preserving maps between the set of bubbling diagrams and the set of marked bumpless pipe dreams. In Section~\ref{sec:supports-of-top}, we prove Theorem~\ref{thm:main2} by showing that bubbling diagrams can be systematically padded to obtain a top-degree diagram which is necessarily in $\supp(\chi_{D^\top(w)})$ (Theorem~\ref{thm:top-diagrams}), and we also show that divisibility relations among monomials in $\mathfrak G_w$ can be realized by inclusion relations among bubbling diagrams in a strong sense (Theorem~\ref{thm:remove-dead-squares}). In Section~\ref{sec:supports-of-homog}, we deduce Theorem~\ref{thm:main3} from a ``one-column version'' of the result (Proposition~\ref{prop:one-column-M-convexity}). In Section~\ref{sec:linear-independence}, we prove Theorem~\ref{thm:main4} and use it to show that our results are sharp.


\section{Background}
\label{sec:background}
\subsection*{Conventions}

We will write permutations $w\in S_n$ in one-line notation as words with the letters $\{1,2,\ldots,n\}$. For example, $w=312\in S_3$ is the permutation that sends $1\mapsto 3$, $2\mapsto 1$, and $3\mapsto 2$. Throughout, permutations will act on the right (switching positions, not values). For $j\in [n-1]$, let $s_j$ denote the adjacent transposition swapping positions $j$ and $j+1$, so for example, $ws_1$ is the permutation $w$ with the numbers $w(1)$ and $w(2)$ swapped. We write $\ell(w)$ for the number of inversions of $w$.

\subsection*{Grothendieck polynomials}

\begin{defn}
Fix $n\geq 1$ and $j \in [n-1]$. The \textbf{divided difference operators} $\partial_j$ are operators on the polynomial ring $\ZZ[x_1, \dots, x_n]$ defined by
\begin{align*}
\partial_j(f) &\overset{\rm def}= \frac{f - s_j\cdot f}{x_j - x_{j+1}} \\&= \frac{f(x_1, \dots, x_n) - f(x_1, \dots, x_{j-1}, x_{j+1}, x_j, x_{j+2}, \dots, x_n)}{x_j - x_{j-1}}.
\end{align*}
The \textbf{isobaric divided difference operators} $\overline \partial_j$ are defined on $\ZZ[x_1,\dots,x_n]$ by
\[
\overline\partial_j(f)\overset{\rm def}=\partial_j(f - x_{j+1}f).
\]
\end{defn}

\begin{defn}
The \textbf{Grothendieck polynomial} $\mathfrak G_w$ of $w \in S_n$ is defined recursively on the weak Bruhat order. Let $w_0$ denote the longest permutation in $S_n$. If $w \neq w_0$, then there is $j \in [n-1]$ with $w(j) < w(j+1)$. The polynomial $\mathfrak G_w$ is defined by
\[
\mathfrak G_w \overset{\rm def}=\begin{cases}
x_1^{n-1}x_2^{n-2}\dots x_{n-1} &\textup{ if } w = w_0,\\
\overline\partial_j\mathfrak G_{ws_j} &\textup{ if } w(j) < w(j+1).
\end{cases}
\]
\end{defn}

Recall that a permutation $w \in S_n$ is \textbf{vexillary} if it is 2143-avoiding, that is, if there do not exist $i < j < k < \ell$ with $w(j) < w(i) < w(\ell) < w(k)$.

\begin{thm}[{\cite{hafner2022}*{Theorem 3.4}}]
\label{thm:HafSupportBetween}
Let $w \in S_n$ be a vexillary permutation, and let $\alpha, \gamma \in \supp(\mathfrak G_w)$ be such that $x^\alpha \mid x^\gamma$. Then, any $\beta\in \ZZ^n$ such that $x^\alpha \mid x^\beta \mid x^\gamma$ is also in $\supp(\mathfrak G_w)$. 
\end{thm}

\subsection*{Marked bumpless pipe dreams}
A \textbf{bumpless pipe dream} (BPD) is a tiling of the $n\times n$ grid with the tiles
\begin{center}
\includegraphics{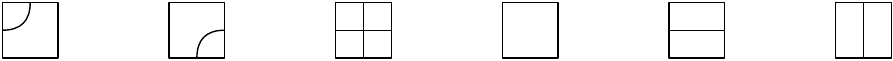}
\end{center}
that form a network of $n$ pipes running from the bottom edge of the grid to the right edge \cites{ls2021, weigandt2021}. A bumpless pipe dream is \textbf{reduced} if each pair of pipes crosses at most once.

Given a bumpless pipe dream, one can define a permutation given by labeling the pipes $1$ through $n$ along the bottom edge and then reading off the labels on the right edge, ignoring any crossings after the first, i.e. replacing redundant crossing tiles \includegraphics{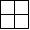} with bump tiles \includegraphics{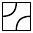}. The set of all bumpless pipe dreams associated to $w$ is denoted $\BPD(w)$ and the set of all reduced bumpless pipe dreams associated to $w$ is denoted $\RBPD(w)$.

For any permutation $w\in S_n$, the \textbf{Rothe bumpless pipe dream} is the unique BPD $P(w) \in \BPD(w)$ which has no up-elbow tiles \scalebox{.6}{\begin{tikzpicture}
\draw[gray, thin] (0,0) rectangle (.5,.5);
\draw[black,thick] (0,.25) .. controls (.25,.25) .. (.25,.5);
\end{tikzpicture}}; each pipe has one down-elbow tile \scalebox{.6}{\begin{tikzpicture}
\draw[gray, thin] (0,0) rectangle (.5,.5);
\draw[black,thick] (.25,0) .. controls (.25,.25) .. (.5,.25);
\end{tikzpicture} } at $(i,w(i))$.

Given $P \in \BPD(w)$, let $D(P)$ denote the set of blank tiles and $U(P)$ denote the up-elbow tiles. A \textbf{marked bumpless pipe dream} (MBPD) is a pair $(P,S)$ where $P \in \BPD(w)$ and $S\subseteq U(P)$. The set of marked bumpless pipe dreams is denoted $\MBPD(w)$.

\begin{prop}[{\cite{weigandt2021}*{Corollary 1.5}}]
We have
\[
\mathfrak G_w = \sum_{(P,S) \in \MBPD(w)}(-1)^{|D(P)| + |S| - \ell(w)}\left(\prod_{(i,j) \in D(P)\cup S}x_i\right).
\]
\end{prop}
Given a MPBD $(P,S)$, the \textbf{weight} of $(P,S)$ is the vector $\wt(P,S)\in\ZZ^n$ whose $i$-th component is the number of tiles in the $i$-th row that are blank or up-elbows.
\begin{cor}
\label{cor:supp-Gw-mbpd}
We have
\[
\supp(\mathfrak G_w) = \{\wt(P,S)\colon (P,S) \in \MBPD(w)\}.
\]
\end{cor}

The \textbf{rank} of a tile $(i,j)\in D(P)$ is the number of pipes northwest of $(i,j)$.

\begin{lem}[{\cite{weigandt2021}}]
\label{lem:vex-reduced}
A permutation $w\in S_n$ is vexillary if and only if every $P\in \BPD(w)$ is reduced.
\end{lem}

A \textbf{local move} is any of the following local transformations of BPDs:
\begin{center}
\includegraphics[scale=0.8]{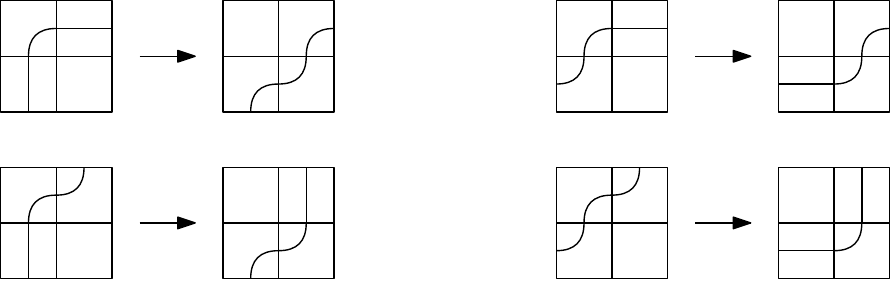}
\end{center}

\begin{lem}[{\cite{weigandt2021}*{Lemma 7.4}, \cite{hafner2022}*{Lemma 2.3}}]
\label{lem:bpd-organized-local-moves}
Let $w\in S_n$ be a vexillary permutation, and let $P \in \BPD(w)$. Then $P$ can be obtained from the Rothe BPD by using local moves to position pipe $w(n)$, then to position pipe $w(n-1)$, and so on through pipe $w(1)$. Alternatively, $P$ can be obtained from the Rothe BPD by using local moves to position pipe $n$, then to position pipe $n-1$, and so on through pipe $1$.
\end{lem}

\subsection*{Diagrams}

As stated in the introduction, a diagram is a subset $D\subseteq[n]\times[k]$. When we draw diagrams, we read the indices as in a matrix.

Associated to any permutation $w \in S_n$ is the \textbf{Rothe diagram} $D(w)\subseteq[n]\times[n]$, defined by
\[
D(w)\overset{\rm def}= \{(i,j)\in[n]\times[n]\colon i < w^{-1}(j) \textup{ and } j < w(i)\}.
\]
The Rothe diagram comes equipped with a \textbf{rank function} $r_{D(w)}\colon D(w)\to\ZZ_{\geq 0}$ defined by
\[
r_{D(w)}(i,j) = |\{(k,w(k))\colon k < i \textup{ and } w(k) < j\}|.
\]

For $R, S\subseteq[n]$, we say $R\leq S$ if $\#R = \#S$ and the $k$-th smallest element of $R$ does not exceed the $k$-th smallest element of $S$ for every $k$. For diagrams $C = (C_1, \dots, C_k)$ and $D = (D_1, \dots, D_k)$, we say $C\leq D$ if $C_j \leq D_j$ for every $j \in [k]$.

\subsection*{Flagged Weyl modules}

Let $Y$ denote a matrix with indeterminates $y_{ij}$ in the upper triangular positions $i \leq j$ and zeroes elsewhere. Given a matrix $M \in M_n(\CC)$ and $R,S\subseteq[n]$, let $M_R^S$ denote the submatrix of $M$ obtained by restricting to rows $S$ and columns $R$.

Let $B$ denote the set of upper triangular matrices in $\GL_n(\CC)$, and let $\mathfrak b$ denote the set of upper triangular matrices in $M_n(\CC)$. The coordinate ring $\CC[\mathfrak b]$ is a polynomial ring in the variables $\{y_{ij}\colon i \leq j\}$. The action of $B$ on $\mathfrak b$ by left multiplication induces an action of $B$ on $\CC[\mathfrak b]$ on the right via $f(Y)\cdot b\colonequals f(b^{-1}Y)$.

The \textbf{flagged Weyl module} of a diagram $D\subseteq[n]\times[k]$ is the subrepresentation
\[
\mathcal M_D \overset{\rm def}=\mathrm{Span}_\CC\left\{\prod_{j=1}^k\det\left(Y_{D_j}^{C_j}\right)\colon C\leq D\right\}
\]
of $\CC[\mathfrak b]$.

The \textbf{dual character} of a representation $M$ of $B$ is the function $\mathrm{char}_M^*\colon T\to\CC$ given by
\[
\mathrm{char}_M^*(\mathrm{diag}(x_1, \dots, x_n)) = \mathrm{tr}(\mathrm{diag}(x_1^{-1},\dots, x_n^{-1})\colon M \to M).
\]
We will write $\chi_D\colonequals\mathrm{char}_{\mathcal M_D}^*$ for the dual character of $\mathcal M_D$.

\begin{prop}[cf.~{\cite{fms2018}*{Theorem 7}}]
\label{prop:support-of-chi}
The function $\chi_D$ equals a polynomial in $\ZZ[x_1, \dots, x_n]$ whose support is $\{\wt(C)\colon C\leq D\}$.
\end{prop}
\begin{proof}
The elements $\prod_j\det(Y_{D_j}^{C_j}) \in \mathcal M_D$ are simultaneous eigenvectors for the action of $T$ with eigenvalue $x^{-\wt(C)}$. Since these elements span $\mathcal M_D$, the dual character is a sum of monomials of the form $x^{\wt(C)}$ for $C\leq D$.
\end{proof}

\subsection*{Generalized permutahedra and M-convexity}

A function $z\colon 2^{[n]}\to\RR$ is called \textbf{submodular} if
\[
z(I) + z(J)\geq z(I\cup J) + z(I\cap J) \qquad\text{ for all } I,J\subseteq[n].
\]
\begin{defn}
A polytope $P\subset\RR^n$ is a \textbf{generalized permutahedron} if there is a submodular function $z\colon 2^{[n]}\to\RR$ such that $z(\emptyset) = 0$ and
\[
P = \left\{t\in\RR^n\colon\sum_{i\in I}t_i\leq z(I) \textup{ for all } I\subseteq[n] \textup{ and } \sum_{i=1}^nt_i = z([n])\right\}.
\]
\end{defn}
\begin{lem}[{\cite{frank2011}*{Theorem 14.2.8}}]
\label{lem:gp-from-rank}
Let $P\subseteq \RR^n$ be a generalized permutahedron defined by a submodular function $z$ with $z(\emptyset) = 0$. Then
\[
z(I) = \max\left\{\sum_{i\in I}p_i\colon p \in P\right\}.
\]
\end{lem}

A set $S\subseteq\ZZ^n$ is \textbf{M-convex} if for any $x,y \in S$ and any $i\in[n]$ for which $x_i > y_i$, there is an index $j\in[n]$ satisfying $x_j < y_j$ and $x - e_i + e_j \in S$ and $y - e_j + e_i \in S$.

Note that the convex hull of an M-convex set is a generalized permutahedron, and the set of integer points of an integer generalized permutahedron is an M-convex set.

\subsection*{Schubert matroid polytopes}
A \textbf{matroid} is a pair $(E,\mathcal B)$ consisting of a finite set $E$ and a nonempty collection of subsets $\mathcal B$ of $E$, called the \textbf{bases} of $M$. The set $\mathcal B$ is required to satisfy the \textbf{basis exchange axiom}: If $B_1, B_2 \in\mathcal B$ and $b_1 \in B_1 \setminus B_2$, then there is $b_2 \in B_2 \setminus B_1$ such that $B_1 \setminus b_1 \cup b_2 \in \mathcal B$. 

\begin{defn}
Fix positive integers $1 \leq s_1 < \dots < s_r \leq n$. The \textbf{Schubert matroid} $\SM_n(s_1, \dots, s_r)$ is the matroid whose ground set is $[n]$ and whose bases are the sets \linebreak $\{a_1, \dots, a_r\}$ with $a_1 < \dots < a_r$ such that $a_1 \leq s_1, \dots, a_r \leq s_r$.
\end{defn}

Given a matroid $M = ([n],\mathcal B)$ and a basis $B \in \mathcal B$, let $\zeta^B = (\zeta^B_1, \dots, \zeta^B_n)$ be the vector with $\zeta^B_i = 1$ if $i \in B$ and $\zeta^B_i = 0$ if $i \not\in B$. The \textbf{matroid polytope} $P(M)$ of $M$ is the convex hull $\mathrm{conv}\{\zeta^B\colon B\in\mathcal B\}$.

The \textbf{rank function} of $M$ is the function $r_M\colon 2^E \to \ZZ_{\geq 0}$ defined by $r_M(S) = \max\{\#(S\cap B)\colon B \in \mathcal B\}$. The function $r_M$ is submodular and $r_M(\emptyset) = 0$. The matroid polytope $P(M)$ is a generalized permutahedron, defined by the submodular function $r_M$.
\section{Bubbling and supports of Grothendieck polynomials}
\label{sec:bubbling-and-supports}
We establish basic properties of bubbling diagrams, including a non-recursive characterization of the set of bubbling diagrams (Lemma~\ref{lem:BD-characterization}), and we prove Theorem~\ref{thm:main1} by constructing weight preserving maps between the set of bubbling diagrams and the set of marked bumpless pipe dreams.
\begin{defn}
\label{defn:bubbling-diagram}
A \textbf{bubbling diagram} is a triple $(D, r, F)$ where $D\subseteq[n]\times[k]$ is a diagram, $r\colon D\to \ZZ_{\geq 0}$ is a function, and $F \subseteq D$ is a collection of squares in $D$ satisfying the following properties:
\begin{itemize}
\item If $(i,j)\in F$, then there exists $(i',j)\in D\setminus F$ with $i' < i$. For the maximal $i' < i$ with $(i',j) \in D\setminus F$, the equality $r(i,j) - r(i',j) = i - i'$ holds.

\item If $(i,j)\in F$ and $(i,k)\in F$, then $r(i,j)\neq r(i,k)$.
\end{itemize}

We refer to the squares $(i,j) \in D\setminus F$ as \textbf{live squares}, to the squares $(i,j) \in F$ as \textbf{dead squares}, and to the squares $(i,j)\not\in D$ as \textbf{empty squares}. (A dead square cannot be bubbled or K-bubbled but impacts the K-bubbleability of certain other squares; see Definitions~\ref{defn:bubbling-move} and~\ref{defn:K-bubbling-move}.)

When we draw a bubbling diagram, the live squares $D\setminus F$ will be colored green, the dead squares $F$ will be colored grey, and each square $(i,j) \in D$ will be labelled by the value of $r(i,j)$. See Figure~\ref{fig:rothebd1423}.

The \textbf{rank} of a square $(i,j)\in D$ is the value $r(i,j)$. The \textbf{weight} of a bubbling diagram $\mathcal D = (D, r, F)$ is $\wt(\mathcal D)\overset{\rm def}=\wt(D)$.
\end{defn}

\begin{example}
The \textbf{Rothe bubbling diagram} of a permutation $w$ is the bubbling diagram $\mathcal D(w)\colonequals(D(w), r_{D(w)}, \emptyset)$.

For example, the Rothe bubbling diagram for $w = 1423$ is shown in Figure~\ref{fig:rothebd1423}.
\begin{figure}[ht]
\includegraphics{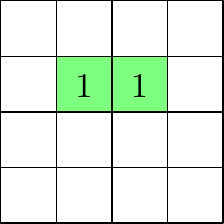}
\caption{The Rothe bubbling diagram for $w = 1423$. All squares in $D(w)$ are live.}
\label{fig:rothebd1423}
\end{figure}
\end{example}

\begin{defn}[Bubbling move]
\label{defn:bubbling-move}
Let $\mathcal D = (D, r, F)$ be a bubbling diagram. Suppose that $(i,j)$ is a live square and that $(i-1,j)$ is an empty square.

Then, a \textbf{bubbling move} at $(i,j)$ produces the bubbling diagram $\mathcal D' = (D', r', F')$ where:
\begin{align*}
D' &\colonequals D\setminus(i,j)\cup(i-1,j)\\
r'(x,y) &\colonequals \begin{cases} r(x,y) &\textup{ if } (x,y) \neq (i-1,j)\\ r(i,j) - 1&\textup{ if } (x,y) = (i-1,j)\end{cases}\\
F' &\colonequals F.
\end{align*}
In other words, we ``bubble up'' a live square $(i,j)$ to $(i-1,j)$, decreasing the rank of the square by $1$ in the process.
\end{defn}
\begin{example}
Let $w = 1423$. The bubbling diagram obtained from $\mathcal D(w)$ by applying a bubbling move at $(2,3)\in D(w)$ is shown in Figure~\ref{fig:bubbled1423}.
\begin{figure}[ht]
\includegraphics{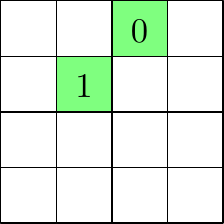}
\caption{A bubbling move applied to $(2,3) \in D(w)$, $w = 1423$}
\label{fig:bubbled1423}
\end{figure}
\end{example}

\begin{defn}[K-bubbling move]
\label{defn:K-bubbling-move}
Let $\mathcal D = (D, r, F)$ be a bubbling diagram. Suppose that $(i,j)$ is a live square and that $(i-1,j)$ is an empty square. Assume furthermore that there are no dead squares $(i,k) \in F$ for which $r(i,j) = r(i,k)$.

Then, a \textbf{K-bubbling move} at $(i,j)$ produces the bubbling diagram $\mathcal D' = (D', r', F')$ where:
\begin{align*}
D' &\colonequals D\cup (i-1,j)\\
r'(x,y)&\colonequals \begin{cases} r(x,y) &\textup{ if } (x,y) \neq (i-1,j)\\ r(i,j) - 1&\textup{ if } (x,y) = (i-1,j)\end{cases}\\
F' &\colonequals F\cup(i,j).
\end{align*}
In other words, we ``bubble up'' a live square $(i,j)$ to $(i-1,j)$, decreasing the rank of the square by $1$ in the process, while also ``leaving behind'' a dead copy of the original square $(i,j)$.
\end{defn}
\begin{example}
Let $w = 1423$. The bubbling diagram obtained from $\mathcal D(w)$ by applying a K-bubbling move at $(2,3)\in D(w)$ is shown in Figure~\ref{fig:Kbubbled1423}.
\begin{figure}[ht]
\includegraphics{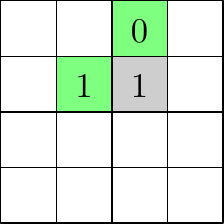}
\caption{A K-bubbling move applied to $(2,3) \in D(w)$, $w = 1423$}
\label{fig:Kbubbled1423}
\end{figure}
\end{example}

\begin{defn}
\label{defn:bubbling-algorithm}
Let $\mathcal D$ be a bubbling diagram. Define $\BD(\mathcal D)$ to be the set of all bubbling diagrams generated from $\mathcal D$ by a series of bubbling moves and K-bubbling moves. For $w \in S_n$, let $\BD(w)\colonequals \BD(\mathcal D(w))$.
\end{defn}
\begin{example}
Let $w = 1423$. Then $\BD(w)$ consists of the bubbling diagrams in Figure~\ref{fig:bd1423}. Note that the two squares in $D(w)$ cannot both be K-bubbled.

\begin{figure}[ht]
\includegraphics{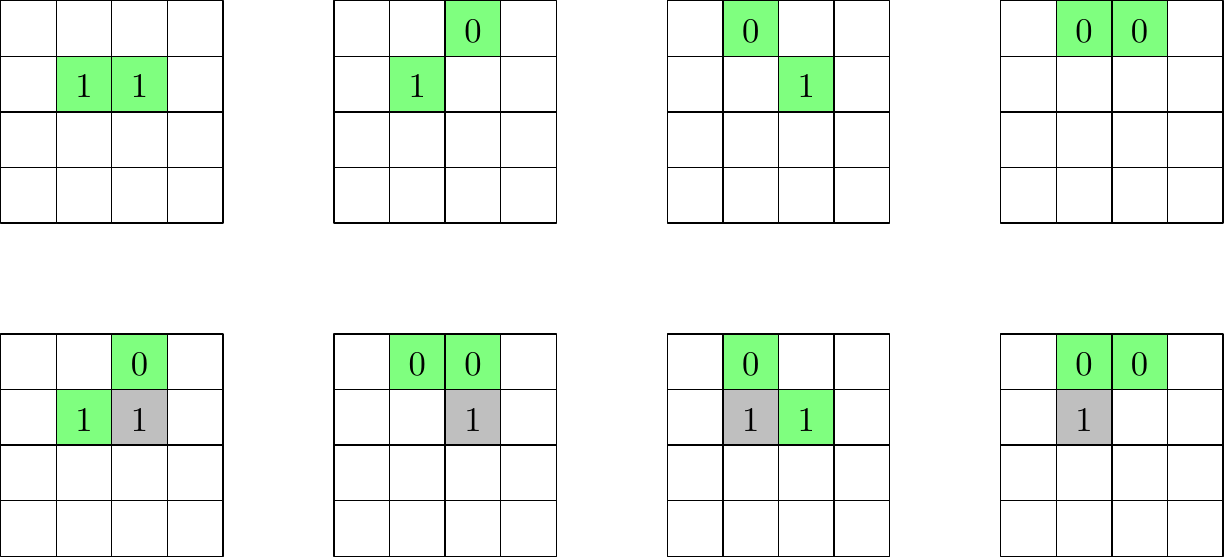}
\caption{The set $\BD(w)$ for $w = 1423$.}
\label{fig:bd1423}
\end{figure}
\end{example}

\begin{defn}
Let $\mathcal D = (D, r, F)$ be a bubbling diagram where $D\subseteq[n]\times[k]$. For $j\in[k]$, define $\mathcal D_j$ to be the $j$-th column of $\mathcal D$, i.e.\ $\mathcal D_j\colonequals (D_j, r|_{D_j}, F_j)$.
\end{defn}
\begin{defn}
Let $D$ be a diagram, and let $r\colon D\to\ZZ_{\geq 0}$ be a function. We say that two squares $(i,j), (i',j')\in D$ are \textbf{linked} if $i - i' = r(i,j) - r(i',j')$. A \textbf{linking class} is an equivalence class of linked squares.
\end{defn}
\begin{example}
Let $w = 178925(10)346$. Then, $\{(6,3), (6,4), (7,6)\}$ is a linking class in $\mathcal D(w)$. See Figure~\ref{fig:linking-class}.
\begin{figure}[ht]
\includegraphics{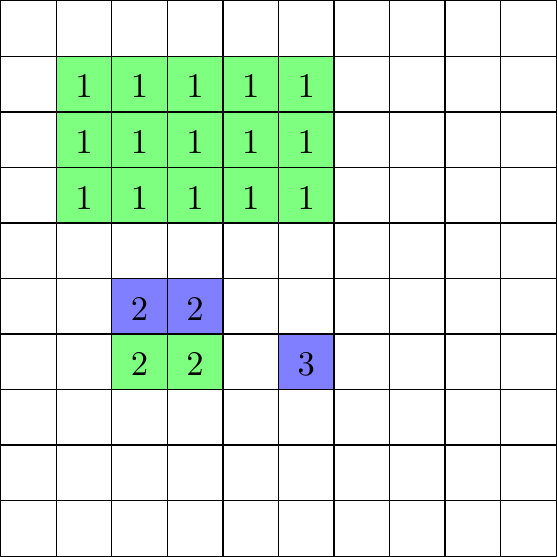}
\caption{The squares in the linking class $\{(6,3), (6,4), (7,6)\}$ are shaded blue.}
\label{fig:linking-class}
\end{figure}
\end{example}
\begin{lem}
\label{lem:propagating-linkedness}
Fix $\mathcal D = (D,r,F)$, and let $\mathcal D' = (D',r',F')\in\BD(\mathcal D)$. If the $a_j$-th highest live square in $D'_j$ is linked to the $a_k$-th highest live square in $D'_k$, then the $a_j$-th highest live square in $D_j$ is linked to the $a_k$-th highest live square in $D_k$.
\end{lem}
\begin{proof}
Let $i_j'$ and $i_j$ denote the $a_j$-th highest live squares in $D'_j$ and $D_j$ respectively, and let $i_k'$ and $i_k$ denote the $a_k$-th highest live squares in $D'_k$ and $D_k$ respectively. By assumption,
\[
r'(i_j',j) - r'(i_k',k) = i_j' - i_k'.
\]
Since $\mathcal D' \in \BD(\mathcal D)$, the equalities
\[
r'(i_j',j) - r(i_j,j) = i_j' - i_j \qquad\textup{ and } \qquad r'(i_k',k) - r(i_k,k) = i_k' - i_k
\]
hold. It follows that
\[
r(i_j,j) - r(i_k,k) = i_j - i_k.\qedhere
\]
\end{proof}
\begin{lem}
\label{lem:vexillary-diagram-column-copies}
Let $w\in S_n$ be a vexillary permutation. Let $(i_1, j_1), (i_2, j_2)\in D(w)$ be two linked squares, and suppose that $j_1 < j_2$. Then:
\begin{enumerate}
\item $i_1 \leq i_2$, and

\item If $(i_1 - 1, j_1), (i_2 - 1, j_2)\not\in D(w)$, then the $j_1$-th and $j_2$-th columns of $D(w)$ agree above the $(i_1-1)$-th row, that is, $(i,j_1) \in D(w)$ if and only if $(i,j_2) \in D(w)$ for all $i \leq i_1-1$ and $r_{D(w)}(i,j_1) = r_{D(w)}(i,j_2)$.
\end{enumerate}
\end{lem}
\begin{proof}
We first show item (1). Suppose that $i_1 > i_2$. Since $j_1 < j_2$, we know
\[
r_{D(w)}(i_2,j_2) \geq |\{(k,w(k))\colon k<i_2\textup{ and } w(k)< j_1\}|.
\]
Thus
\begin{align*}
r_{D(w)}(i_2, j_2) - r_{D(w)}(i_1, j_1) &\geq|\{(k,w(k))\colon k < i_2 \textup{ and } w(k) < j_1\}| \\
&\phantom{aaaaaaaaa}-|\{(k,w(k))\colon k<i_1\textup{ and } w(k) < j_1\}|\\
&= -|\{(k,w(k))\colon i_2 \leq k < i_1 \textup{ and } w(k) < j_1\}| \\
&\geq i_2 - i_1.
\end{align*}
If equality occurs, then $w(k) < j_1$ for all $i_2 \leq k < i_1$; in particular, $w(i_2) < j_1 < j_2$. This contradicts the fact that $(i_2,j_2)\in D(w)$.

We now show item (2). Because $(i_1,j_1)\in D(w)$, we know that $(i,j_2)\in D(w)$ implies $(i,j_1) \in D(w)$ for all $i < i_1$. If $(i,j_1)\in D(w)$ and $(i,j_2)\not\in D(w)$, then $i < i_1 - 1 < i_2 < w^{-1}(j_2)$ forms a $2143$ pattern.
\end{proof}
\begin{defn}
Fix a bubbling diagram $\mathcal D = (D, r, F)$. Let $F'\subseteq D'$ be diagrams. We say that $(D', F')$ is \textbf{$\mathcal D$-admissible} if:
\begin{enumerate}
\item $F'\supseteq F$,

\item $D'\setminus F' \leq D \setminus F$,

\item For any $(i,j) \in F'\setminus F$, there is $(i',j)\in D'\setminus F'$ with $i' < i$. If the $m$-th highest live square in $D'_j\setminus F'_j$ is the live square that is immediately above $(i,j)$, then the $m$-th highest live square in $D_j\setminus F_j$ is below row $i$,

\item Suppose that $(i,j)\in F$ and that there are $a$ live squares in $D_j$ above row $i$ and $b$ live squares in $D_j$ below row $i$. Then there are $a$ live squares in $D'_j$ above row $i$ and $b$ live squares in $D'_j$ below row $i$,

\item Let $(i,j), (i,k)\in F'$ be dead squares in the same row. Suppose that the $m_j$-th highest square in $D'_j\setminus F'_j$ is the live square immediately above $(i,j)$ and that the $m_k$-th highest square in $D'_k \setminus F'_k$ is the live square immediately above $(i,k)$. Then the $m_j$-th highest square in $D_j$ and the $m_k$-th highest square in $D_k$ are not linked.
\end{enumerate}
\end{defn}
In Lemma~\ref{lem:BD-characterization}, we will show that $\mathcal D$-admissibility is equivalent to membership in $\BD(\mathcal D)$. To get there, we describe a systematic way to generate a given bubbling diagram $\mathcal D'\in\BD(\mathcal D)$ (Definition~\ref{defn:canonical-bubbling-sequence}); the legality of the construction is the content of Lemma~\ref{lem:canonical-bubbling-sequence}.

\begin{defn}
Let $\mathcal D = (D,r, F)$ be a bubbling diagram, and let $F'\subseteq D'$ be a pair of diagrams. Let $\mathcal D' = (D', F')$ and $\mathcal D'_j = (D'_j, F'_j)$. We say that $\mathcal D_j$ and $\mathcal D'_j$ \textbf{weakly agree below row $s$} if:
\begin{itemize}
\item $D_j\cap\{s, \dots, n\} = D'_j \cap \{s, \dots,n\}$, and

\item $F_j\cap\{s, \dots,n\} = F'_j\cap\{s, \dots,n\}$.
\end{itemize}
We write $s(\mathcal D_j, \mathcal D'_j)$ to mean the minimal integer $s$ so that $\mathcal D_j$ and $\mathcal D'_j$ weakly agree below row $s$. If no such row exists, then we set $s(\mathcal D_j, \mathcal D'_j)\colonequals n+1$.
\end{defn}
\begin{defn}
\label{defn:canonical-bubbling-sequence}
Let $\mathcal D = (D, r, F)$ be a bubbling diagram. Suppose that $F'\subseteq D'$ are $\mathcal D$-admissible. The \textbf{canonical bubbling sequence} of $(D',F')$ with respect to $\mathcal D$ is the sequence $(\mathcal D^n, \dots, \mathcal D^0)$ defined by:
\begin{itemize}
\item $\mathcal D^n \colonequals\mathcal D$

\item For $m \geq 1$, $\mathcal D^{m-1}$ is obtained from $\mathcal D^m = (D^m, r^m, F^m)$ by applying the following bubbling and K-bubbling moves. For each column $j$ with an empty square in row $m$ directly above a live square in row $m+1$, let $k_j\colonequals s(\mathcal D^m_j,\mathcal D'_j)$. If $k_j > m+1$ and $k_j-1 \not\in D'_j$, then apply bubbling moves at $m+1,m+2, \dots, k_j-1$; if $k_j > m+1$ and $k_j -1 \in F'_j$, then apply bubbling moves at $m+1, m+2, \dots, k_j-2$ and then a K-bubbling move at $k_j-1$. 
\end{itemize}
See Figure~\ref{fig:canonical-bubbling-sequence} for an example.
\end{defn}
\begin{figure}[ht]
\includegraphics[scale=0.7]{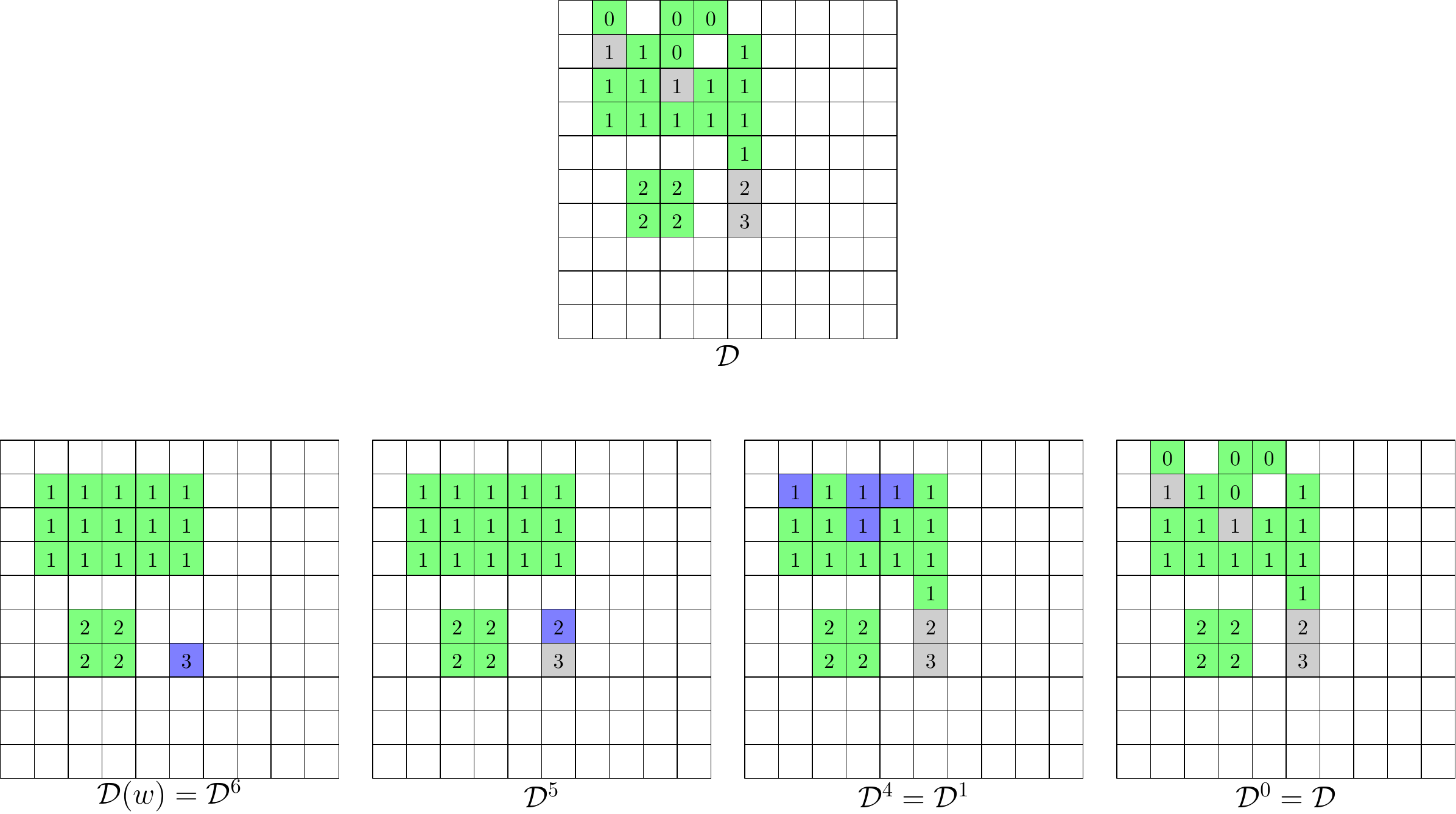}
\caption{The canonical bubbling sequence for a bubbling diagram $\mathcal D\in\BD(w)$ with $w = 178925(10)346$. The squares which are bubbled or K-bubbled are shaded blue.}
\label{fig:canonical-bubbling-sequence}
\end{figure}
To ensure the legality of the bubbling moves in Definition~\ref{defn:canonical-bubbling-sequence}, we use the following lemma.
\begin{lem}
\label{lem:canonical-bubbling-sequence}
Let $\mathcal D = (D,r,F)$ be a bubbling diagram. Suppose that $F'\subseteq D'$ is $\mathcal D$-admissible, and let $(\mathcal D^n, \dots, \mathcal D^0)$ denote the canonical bubbling sequence. Write $\mathcal D^m = (D^m, r^m, F^m)$. Suppose that for some $m$ and $j$, we have $(m,j)\not \in D^m$ and $(m+1,j) \in D^m\setminus F^m$. Let $\ell_{j,m} \geq m+1$ be maximal so that $(m+1,j), \dots, (\ell_{j,m},j) \in D^m\setminus F^m$. Let $k_{j,m}\colonequals s(\mathcal D^m_j, \mathcal D'_j)$. Then:
\begin{enumerate}
\item $\ell_{j,m} \geq k_{j,m} - 1$,

\item Suppose that $k_{j,m}-1\in F'_j$ and that for some $h\neq j$, either $k_{j,m}-1\in F^m_h$ or $m\not\in D^m_h$, $m+1\in D^m_h$, $k_{h,m}=k_{j,m}$, and $k_{h,m}-1 \in F'_h$. Then $r^m(k_{j,m}-1,j)\neq r^m(k_{h,m}-1,h)$.
\end{enumerate}
Furthermore, $D^0 = D'$ and $F^0 = F'$.
\end{lem}
\begin{proof}
We first show item (1) using induction. Let $m_1 > \dots > m_h$ be the integers for which $\mathcal D^{m_i}_j \neq \mathcal D^{m_i-1}_j$. Thus, $m_1$ is maximal so that $m_1\not\in D_j$ and $m_1+1\in D_j\setminus F_j$. It follows that if $i > \ell_{j,m_1}$ and $i\in D_j\setminus F_j$, then there is $m'<i$ with $m'\in F_j$ and $m'+1, m'+2, \dots, i \in D_j\setminus F_j$. Condition (1) implies that $m'\in F'_j$, and conditions (2) and (4) together imply
\[
(D_j\setminus F_j)\cap\{\ell_{j,m_1}+1, \dots, n\} = (D'_j\setminus F'_j)\cap\{\ell_{j,m_1}+1,\dots,n\}.
\]
Then condition (3) implies that $(F'\setminus F)\cap\{\ell_{j,m_1}+1,\dots,n\} = \emptyset$. We conclude that $\ell_{j,m_1}\geq k_{j,m_1} - 1$.

Now assume that $\ell_{j,m_i} \geq k_{j,m_i} - 1$. By construction of the canonical bubbling sequence, $\mathcal D^{m_{i+1}} = \mathcal D^{m_i-1}_j$ weakly agrees with $\mathcal D'_j$ below row $k_{j,m_i}-1$. Furthermore, $k_{j,m_i}-1\not\in D'_j\setminus F'_j$, so $\ell_{j,m_{i+1}}\leq k_{j,m_i}-2$. Thus 
\[
k_{j,m_{i+1}}-1 \leq k_{j,m_i}-2 \leq \ell_{j,m_{i+1}}.
\]

We now show item (2). By induction, we may assume that $\mathcal D^m\in\BD(\mathcal D)$. Suppose that $k_{j,m}\in D^m_j\setminus F^m_j$ is the $a_j$-th highest live square in its column. Define $a_h$ as follows: if $(k_{j,m}-1,h)\in F^m$, then suppose that the $a_h$-th highest square in $D^m_h\setminus F^m_h$ is the live square immediately above $(k_{j,m}-1,h)$; if $k_{j,m}-1\in D^m_h\setminus F^m_h$, then suppose that $k_{j,m}-1$ is the $a_h$-th highest live square in its column. If $r^m(k_{j,m}-1,j)\neq r^m(k_{h,m}-1,h)$, then the $a_j$-th highest live square in $D^m_j$ and the $a_h$-th highest live square in $D^m_h$ are linked. Lemma~\ref{lem:propagating-linkedness} implies that the $a_j$-th highest live square in $D_j$ and the $a_h$-th highest live square in $D_h$ are linked, contrary to condition (5) in Lemma~\ref{lem:BD-characterization}.

We now show that $D^0 = D'$ and $F^0 = F'$. We claim that $F'\subseteq D'$ is $\mathcal D^0_j$-admissible, that is:
\begin{enumerate}
\item The canonical bubbling sequence introduces a dead square $(i,j)\in F^0\setminus F$ only if $(i,j)\in F'\setminus F$, so $F'\supseteq F^0$,

\item The canonical bubbling sequence bubbles the $m$-th highest live square $i\in D^m_j$ only if the $m$-th highest live square in $D'_j$ is above row $i$, so $D'\setminus F'\leq D^0\setminus F^0$,

\item Since $F'\supseteq F^0$, for any $(i,j)\in F'\setminus F^0$, there is $(i',j)\in D'\setminus F'$ with $i' < i$. If the $m$-th and $(m+1)$-th highest live squares in $D^0_j\setminus F^0_j$ are in row $k$ and $\ell$ respectively, then $F^0_j\cap\{k+1,\dots,\ell-1\} = F'_j \cap\{k+1,\dots,\ell-1\}$; thus, if the $m$-th highest live square in $D'_j\setminus F'_j$ is the live square that is immediately above $(i,j)$, then the $m$-th highest live square in $D^0_j\setminus F^0_j$ is below row $i$.

\item Suppose that $(i,j)\in F^0$ and that there are $a$ live squares in $D^0_j$ above row $i$ and $b$ live squares in $D^0_j$ below row $i$. If $(i,j)\in F^m\setminus F^{m+1}$ for some $m$, then $i+1 = k_{j,m+1}$ and there are $a$ and $b$ live squares in $D^m_j$ above and below row $i$ respectively. It follows that there are $a$ and $b$ live squares in $D'_j$ above and below row $i$ respectively.

\item Let $(i,j),(i,k)\in F'$ be dead squares in the same row. Suppose that the $m_j$-th highest square in $D'_j\setminus F'_j$ is the live square immediately above $(i,j)$ and that the $m_k$-th highest square in $D'_k \setminus F'_k$ is the live square immediately above $(i,k)$. Then the $m_j$-th highest square in $D_j$ and the $m_k$-th highest square in $D_k$ are not linked; thus Lemma~\ref{lem:propagating-linkedness} guarantees that the $m_j$-th highest square in $D^0_j$ and the $m_k$-th highest square in $D^0_k$ are not linked.
\end{enumerate}
Let $m$ be minimal so that $m\not\in D^0_j$ and $m+1\in D^0_j\setminus F^0_j$; if no such $m$ exists, then set $m\colonequals n+1$. It follows that if $i < m$ and $i\in D^0_j\setminus F^0_j$, then there is $m' < i$ with $m'\in F^0_j$ and $m'+1, \dots, i\in D^0_j\setminus F^0_j$. Condition (1) implies that $m'\in F'_j$, and conditions (2) and (4) together imply $(D^0_j\setminus F^0_j)\cap\{1,\dots,m-1\} = (D'_j\setminus F'_j)\cap\{1,\dots,m-1\}$. Then condition (3) implies that $(F'\setminus F)\cap\{1,\dots,m-1\} = \emptyset$, so it suffices to show that $\mathcal D^0_j$ and $\mathcal D'_j$ agree below row $m+1$.

Either $m\not\in D^m_j$, $m+1\in D^m_j$, and $s(\mathcal D^m_j, \mathcal D'_j) \leq m+1$ or $m, m+1\in D^m_j$ and there exists $\ell < m$ with $k_{j,\ell} = m+1$ so that $s(\mathcal D^\ell_j,\mathcal D'_j) \leq m$.  Furthermore, in both cases, the canonical bubbling sequence then leaves rows $\{m+1, \dots, n\}$ invariant. It follows that $s(\mathcal D^0_j,\mathcal D'_j) \leq m+1$, and we conclude that $D^0_j = D'_j$ and $F^0_j = F'_j$.
\end{proof}

\begin{lem}
\label{lem:BD-characterization}
Let $\mathcal D = (D, r, F)$ be a bubbling diagram, and let $F'\subseteq D'$ be diagrams. Then there exists $r'\colon D'\to\ZZ_{\geq 0}$, necessarily unique, so that $(D',r',F')\in\BD(\mathcal D)$ if and only if $(D', F')$ is $\mathcal D$-admissible.
\end{lem}
\begin{proof}[Proof of Lemma~\ref{lem:BD-characterization}]
A straightforward check confirms that if $\mathcal D'\in\BD(\mathcal D)$ is $\mathcal D$-admissible, then any bubbling diagram $\mathcal D''$ obtained from $\mathcal D'$ via a bubbling or K-bubbling move is $\mathcal D$-admissible. Thus, the forward implication follows.

Conversely, take any $\mathcal D$-admissible $F'\subseteq D'$. The canonical bubbling sequence of $(D',F')$ with respect to $\mathcal D$ gives a bubbling diagram $\mathcal D^0 = (D^0, r^0, F^0) \in\BD(\mathcal D)$, and Lemma~\ref{lem:canonical-bubbling-sequence} guarantees that $D^0 = D'$ and $F^0 = F'$.
\end{proof}
\newtheorem*{thm:main1}{Theorem~\ref{thm:main1}}
\begin{thm:main1}
Let $w \in S_n$ be a vexillary permutation. Then $\supp(\mathfrak G_w) = \{\wt(\mathcal D)\colon \mathcal D \in \BD(w)\}$.
\end{thm:main1}
\begin{proof}[Proof of Theorem~\ref{thm:main1}]
We first show that $\supp(\mathfrak G_w)\subseteq\{\wt(\mathcal D)\colon \mathcal D\in\BD(w)\}$. By Corollary~\ref{cor:supp-Gw-mbpd}, it suffices to show that for every marked bumpless pipe dream $(P,S) \in \MBPD(w)$, there exists a bubbling diagram $\mathcal D\in\BD(w)$ so that $\wt(\mathcal D) = \wt(P,S)$.  We will construct such a bubblinng diagram $\mathcal D$ as follows; see Example~\ref{example:Dm-from-Pm} for an example.

Fix $(P,S) \in \MBPD(w)$. Let $(P_m, S_m) \in \MBPD(w)$ denote the MBPD whose $w(m+1), w(m+2), \dots, w(n)$-th marked pipes agree with those of $(P,S)$ and whose $w(1),$ $w(2),$ $\dots,$ $w(m)$-th marked pipes agree with those of the Rothe BPD $P(w)$. Let $B_m$ denote the set of blank tiles of $P_m$ which are not southeast of any of the pipes $w(m+1), \dots, w(n)$.

We will use induction to construct diagrams $\mathcal D^m = (D^m, r^m, F^m) \in\BD(w)$ and bijections $f_m\colon[n]\to[n]$ so that:
\begin{itemize}
\item The diagram $\mathcal D^m$ agrees with $\mathcal D(w)$ above row $m$,

\item $\wt(\mathcal D^m) = \wt(P_m, S_m)$,

\item $(i,j)\in B_m$ if and only if $(i, f_m(j))\in D^m\setminus F^m$ for every $1 \leq i \leq \max\{k\colon (k,j)\in B_m\}$, and

\item The rank of any square $(i,j) \in B_m$ is equal to $r^m(i,f_m(j))$.
\end{itemize}

Since $(P_n, S_n)$ is the Rothe BPD, we set $\mathcal D^n\colonequals \mathcal D(w)$ and $f_n \colonequals \mathrm{id}$; the three items above hold because $B_n = D(w)$ (as subsets of $[n]\times[n]$) and $S_n = F_n = \emptyset$.

If $(P_{m-1}, S_{m-1}) = (P_m, S_m)$, then we define $\mathcal D^{m-1}\colonequals\mathcal D^m$ and $f_{m-1}\colonequals f_m$.

Now suppose $(P_{m-1},S_{m-1})\neq(P_m, S_m)$. Assume we are given $\mathcal D^m$ and $f_m$ satisfying the items above. Let $Q_m$ denote the set of blank tiles in $B_m$ which are displaced upon replacing the $w(m)$-th pipe in $P_m$ with the $w(m)$-th (marked) pipe of $P$ to obtain $(P_{m-1}, S_{m-1})$.

Any square $(i,j) \in Q_m$ that is northernmost in its column satisfies $i = m + 1$; it follows by construction of $f_m$ that $(m+1,f_m(j)) \in D_m\setminus F_m$ and that $(m, f_m(j))\not \in D_m$. Furthermore, any square in $Q_m$ has the same rank, $Q_m\cap S_{m-1}$ has at most one square in each row, and any square in $Q_m\cap S_{m-1}$ is southernmost in its column. Thus, we may apply bubbling moves to $\mathcal D_m$ at the squares $\{(i, f_m(j))\colon (i,j) \in Q_m \setminus S_{m-1}\}$ followed by K-bubbling moves at the squares $\{(i,f_m(j))\colon (i,j) \in Q_m\cap S_{m-1}\}$ to produce a bubbling diagram $\mathcal D_{m-1}$.  This bubbling diagram agrees with $\mathcal D(w)$ above row $m-1$ and satisfies $\wt(\mathcal D_{m-1}) = \wt(P_{m-1}, S_{m-1})$. 

It remains to define the bijection $f_{m-1}$.  Let $j_1 < \dots < j_k$ denote the columns which have squares in $Q_m$, and set $j_0\colonequals j_1-1$. Let $\sigma_m\colon\{j_0, \dots,j_k\}\to\{j_0, \dots, j_k\}$ be the map $\sigma_m(j_\ell) = j_{\ell+1}$ (with $j_{k+1}\colonequals j_0$), and let $f_{m-1} \colonequals f_m\circ\sigma_m$. Since $(m,j_0), \dots, (m,j_k) \not \in D(w)$, Lemma~\ref{lem:vexillary-diagram-column-copies} guarantees that the $j_0, \dots, j_k$-th columns of $D(w)$ all agree above row $m$. Combined with the fact that $(i,j) \in B_m$ if and only if $(i,f_m(j)) \in D^m\setminus F^m$ for all $1 \leq i \leq \max\{k\colon(k,j)\in B_m\}$, it follows that $(i,j) \in B_{m-1}$ if and only if $(i, f_{m-1}(j))\in D^{m-1}\setminus F^{m-1}$ for every $1 \leq i \leq \max\{k\colon (k,j) \in B_{m-1}\}$ and furthermore, that the rank of any square $(i,j) \in B_{m-1}$ is equal to $r^{m-1}(i,f_{m-1}(j))$.\\

We now show that $\supp(\mathfrak G_w)\supseteq\{\wt(\mathcal D)\colon \mathcal D\in\BD(w)\}$. By Corollary~\ref{cor:supp-Gw-mbpd}, it suffices to show that for every diagram $\mathcal D\in\BD(w)$, there exists a marked bumpless pipe dream $(P,S)\in\MBPD(w)$ such that $\wt(\mathcal D) = \wt(P,S)$. We accomplish this using the following construction, see Example~\ref{example:Pm-from-Dm} for an example.

Fix $\mathcal D = (D, r, F) \in \BD(w)$. Let $\mathcal D^m = (D^m, r^m, F^m)$ denote the canonical bubbling sequence. We will construct MBPDs $(P_m, S_m)$ and bijections $g_m\colon[n]\to[n]$ so that:
\begin{itemize}
\item The $w(1), \dots, w(m)$-th pipes of $P_m$ have no up-elbow tiles;

\item $\wt(\mathcal D^m) = \wt(P_m, S_m)$;

\item $(i,j) \in B_m$ if and only if $(i, g_m(j)) \in D^m\setminus F^m$ for every $1\leq i \leq \max\{k\colon(k,j)\in B_m\}$;

\item The rank of any square $(i,j) \in B_m$ is equal to $r_m(i,f_m(j))$.
\end{itemize}
Since $\mathcal D^n$ is the Rothe bubbling diagram, we set $(P_n, S_n)$ to be the Rothe BPD and $g_n \colonequals \mathrm{id}$. If $\mathcal D^{m-1} = \mathcal D^m$, then we define $(P_{m-1}, S_{m-1}) \colonequals (P_m, S_m)$ and $f_{m-1}\colonequals f_m$.

Now suppose $\mathcal D^{m-1}\neq\mathcal D^m$. Assume we are given $(P_m, S_m)$ and $f_m$ satisfying the items above. Let $j_1, \dots, j_\ell$ be the columns 
which are bubbled when constructing $\mathcal D^{m-1}$, and let $k_{j_i} = s(\mathcal D^m_{j_i},\mathcal D'_{j_i})$, indexed so that $k_{j_1} \geq \dots \geq k_{j_\ell}$ and so that if $(k_{j_i},j_i) \in F$, then $k_{j_i} > k_{j_{i+1}}$.

By assumption on $g_m$ and by Lemma~\ref{lem:canonical-bubbling-sequence}, the squares $(x, g_m^{-1}(j_i))$ are in $B_m$ for $m+1 \leq x \leq k_{j_i}-1$ while the squares $(m,g_m^{-1}(j_i))$ are not in $B_m$. It follows that the squares $(i,g_m^{-1}(j_i))$ are southeast of pipe $w(m)$.

Let
\[
R_m = \{(x,w^{-1}(m)+y)\colon m+1\leq x \leq k_{j_y}-1, 1 \leq y \leq \ell\}.
\]
We define $(P_{m-1}, S_{m-1})$ to be the BPD obtained from $P_m$ by replacing pipe $w(m)$ with the pipe that traces the southeasternmost squares of $R_m$ and marking the up-elbow tiles at $(k_{j_i}-1, w^{-1}(m))$ whenever $(k_{j_i}-1, j_i) \in F$. The $w(1), \dots, w(m-1)$-th pipes of $P_{m-1}$ have no up-elbow tiles, and this BPD satisfies $\wt(\mathcal D_{m-1}) = \wt(P_{m-1},S_{m-1})$. 

It remains to define the bijection $g_{m-1}$.  Let $h_m$ be the maximal integer such that $(m+1,h_m)\in B_m$. Let $\psi_m\colon\{w^{-1}(m), w^{-1}(m)+1, w^{-1}(m)+2, \dots, h_m\}\to\{w^{-1}(m), w^{-1}(m)+1, w^{-1}(m)+2, \dots, h_m\}$ denote a fixed bijection which sends $w^{-1}(m)+i-1$ to $g_m^{-1}(j_i)$. Then we define $g_{m-1}\colonequals g_m\circ\psi_m$. Since $(m,w^{-1}(m)+1), \dots, (m, h_m) \not\in D(w)$, Lemma~\ref{lem:vexillary-diagram-column-copies} guarantees that the $w^{-1}(m)+1, \dots, h_m$-th rows of $D(w)$ all agree above row $m$. Combined with the fact that $(i,j)\in B_m$ if and only if $(i,g_m(j))\in D^m\setminus F^m$ for all $1\leq i\leq\max\{k\colon(k,j)\in B_m\}$, it follows that $(i,j)\in B_{m-1}$ if and only if $(i, g_{m-1}(j))\in D^{m-1}\setminus F^{m-1}$ for every $1\leq i\leq\max\{k\colon(k,j)\in B_{m-1}\}$ and furthermore, that the rank of any square $(i,j)\in B_{m-1}$ is equal to $r^{m-1}(i,g_{m-1}(j))$.
\end{proof}
\begin{example}
\label{example:Dm-from-Pm}
Let $(P,S)\in \MBPD(w)$ be as in Figure~\ref{fig:PS-Dm-from-Pm}.
\begin{figure}[ht]
\includegraphics[scale=0.8]{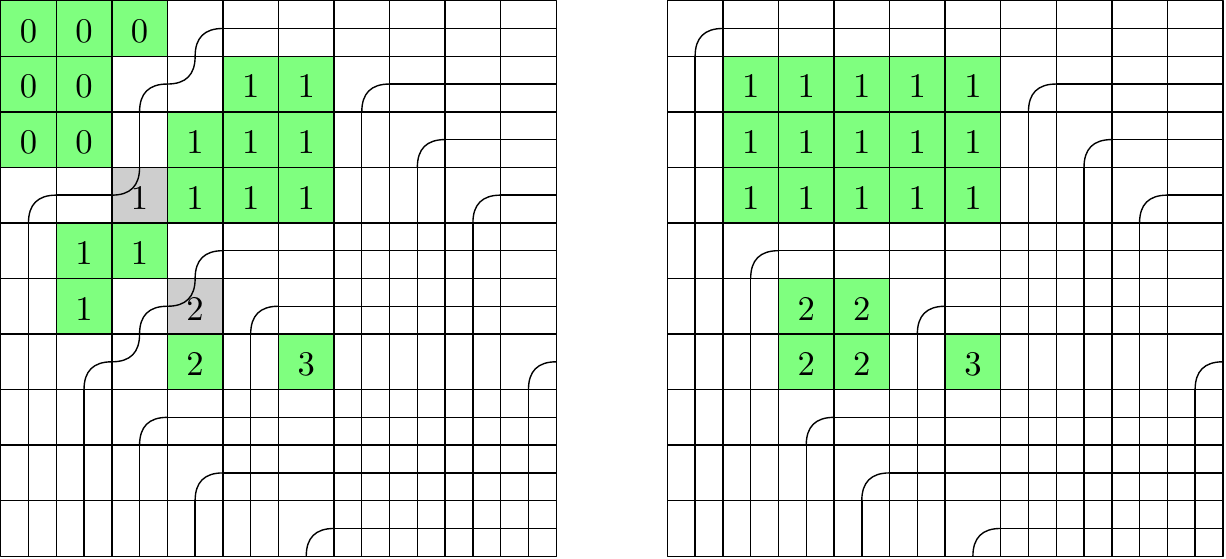}
\caption{Left: A marked pipe dream $(P,S) \in \MBPD(w)$ for $w=178925(10)346$ with blank tiles shaded green and marked up-elbows shaded grey; each blank and marked tile is labeled by its rank. Right: The Rothe MBPD $(P(w), \emptyset)$ with blank and marked tiles labeled by their ranks.}
\label{fig:PS-Dm-from-Pm}
\end{figure}

The MBPDs $(P_m,S_m)$ are equal to $(P(w),\emptyset)$ for $m\geq 5$, so the bubbling diagrams $\mathcal D^m$ are equal to $\mathcal D(w)$ for $m\geq 5$. Furthermore, $f_m$ is the identity for $m\geq 5$. The set 
\[
Q_5 = \{(6,3),(7,3), (6,4)\}
\]
is, however, nonempty, so $(P_4, S_4) \neq (P_5, S_5)$. The set $Q_5$ contains squares in columns $3$ and $4$, and the map $f_4 = \sigma_5$ cyclically permutes the set $(2,3,4)$.

Then $(P_m, S_m) = (P_4, S_4)$ for $1\leq m\leq 4$, so $\mathcal D^m = \mathcal D^4$ for $1\leq m\leq 4$; however, 
\[
Q_1 = \{(2,2), (3,2), (4,2), (2,3), (3,3), (4,3), (2,4)\}
\]
is nonempty, so $(P_0, S_0)\neq (P_1, S_1)$. The set $Q_1$ contains squares in columns $2$, $3$, and $4$, and the map $\sigma_1$ cyclically permutes the set $(1,2,3,4)$. Thus, the map $f_0 = f_4\circ \sigma_1$ cyclically permutes the set $(1,3,2,4)$. See Figure~\ref{fig:Qmfm}.

\begin{figure}[ht]
\includegraphics[scale=0.7]{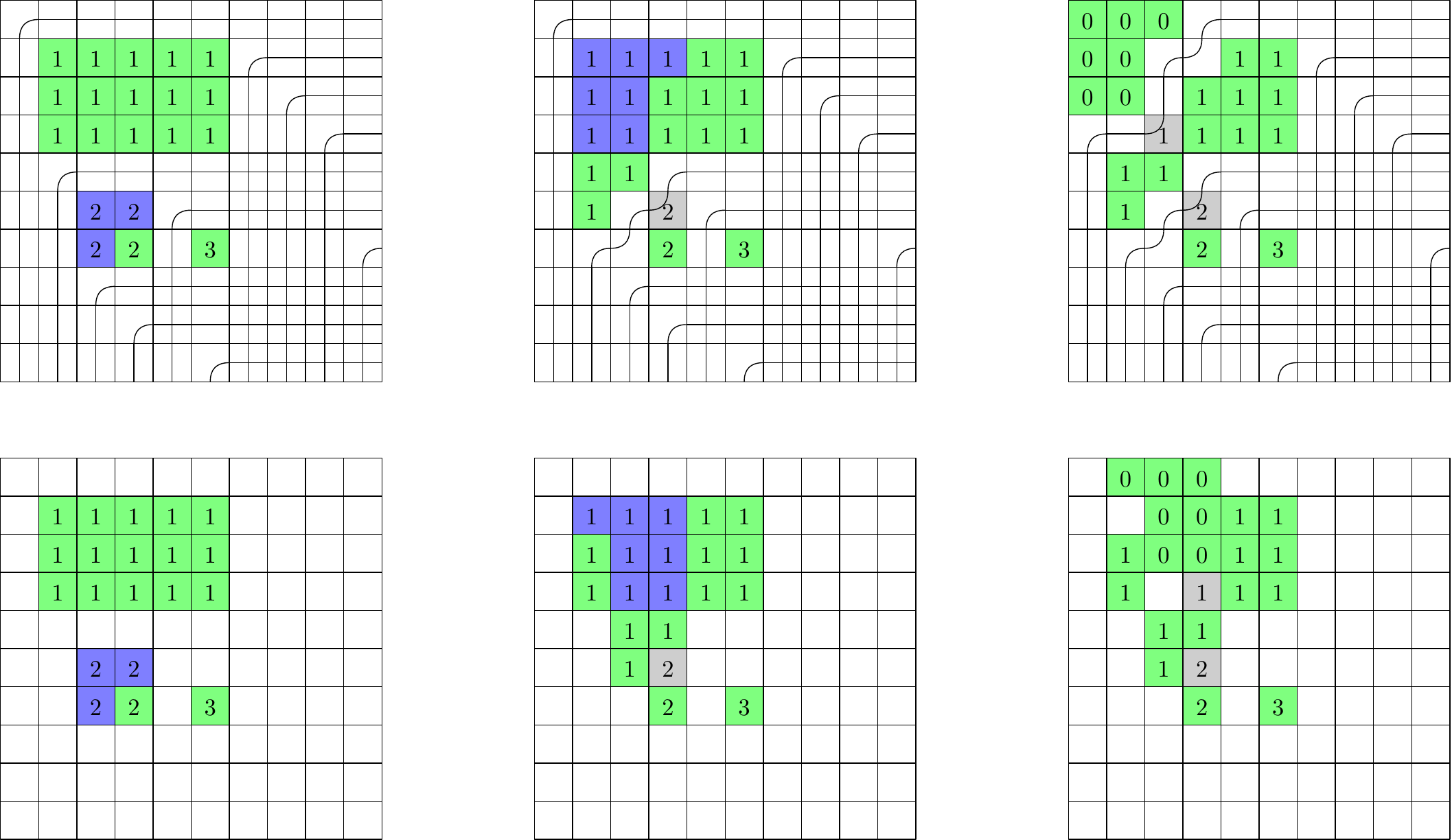}
\caption{Left column: The pipe dream $(P_5,S_5) = (P(w),\emptyset)$ for $P$ as in Figure~\ref{fig:PS-Dm-from-Pm} along with the diagram $\mathcal D^5 = \mathcal D(w)$. The squares in $Q_5$ and its image under $f_5$ are shaded blue. Middle column: The pipe dream $(P_1,S_1)$ along with $\mathcal D^1$. The squares in $Q_1$ and its image under $f_1$ are shaded blue. Right column: The pipe dream $(P_0,S_0) = P$ along with the diagram $\mathcal D^0 = \mathcal D$.}
\label{fig:Qmfm}\end{figure}
\end{example}
\begin{example}
\label{example:Pm-from-Dm}
Let $\mathcal D\in\BD(w)$ be as in Figure~\ref{fig:D-Pm-from-Dm}.
\begin{figure}[ht]
\includegraphics[scale=0.8]{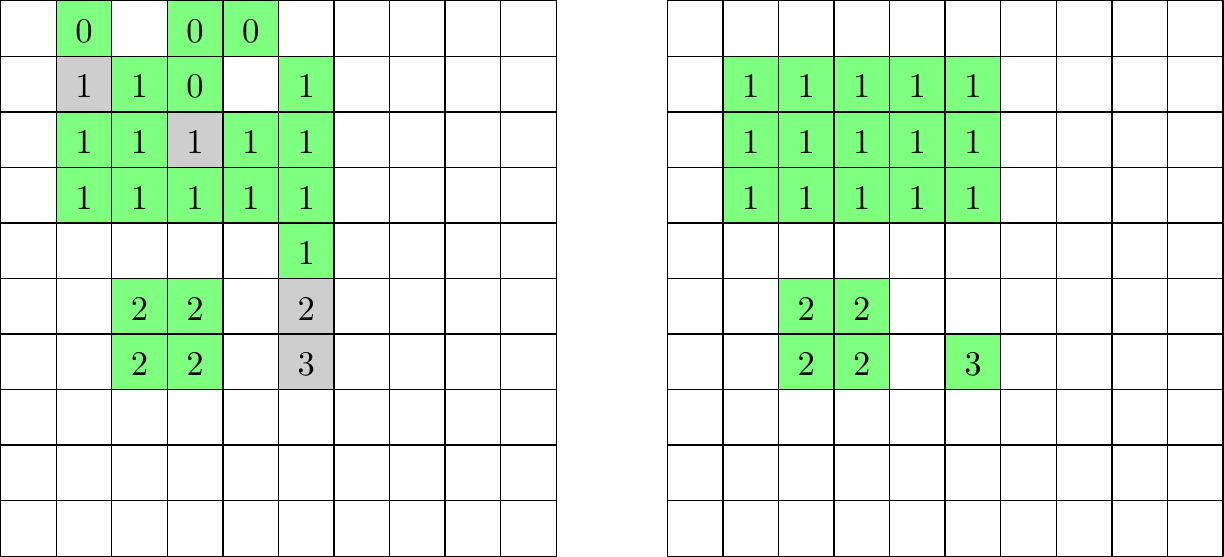}
\caption{Left: A bubbling diagram $\mathcal D \in \BD(w)$ for $w=178925(10)346$. Right: The Rothe bubbling diagram $\mathcal D(w)$.}
\label{fig:D-Pm-from-Dm}
\end{figure}
The bubbling diagrams $\mathcal D^m$ are equal to $\mathcal D(w)$ for $m\geq 6$, so the MBPDs $(P_m,S_m)$ are equal to $(P(w),\emptyset)$ for $m\geq 6$. Furthermore, $g_m$ is the identity for $m\geq 6$. Then, $\mathcal D^5$ is obtained from $\mathcal D^6$ by applying a K-bubbling move to $(7,6)\in D^6\setminus F^6$. We have
\[
R_6 = \{(7,6)\},
\]
so $(P_5, S_5) \neq (P_6, S_6)$. The map $\psi_6\colon\{5,6\}\to\{5,6\}$ is a fixed bijection which sends $w^{-1}(6)+1-1 = 5$ to $g_6^{-1}(6) = 6$, so $\psi_6$ will cyclically permute $(5,6)$. Thus, $g_5 = \psi_6$ cyclically permutes $(5,6)$.

Next, $\mathcal D^4$ is obtained from $\mathcal D^5$ by applying a K-bubbling move to $(6,6)\in D^5\setminus F^5$. We have
\[
R_5 = \{(6,3)\},
\]
so $(P_4, S_4)\neq (P_5, S_5)$. The map $\psi_5\colon\{2,3,4,5\}\to\{2,3,4,5\}$ is a fixed bijection which sends $w^{-1}(5)+1-1 = 2$ to $g_5^{-1}(6) = 5$; suppose that $\psi_5$ cyclically permutes $(2,5)$. Then, $g_4 = g_5\circ\psi_5$ cyclically permutes $(2,6,5)$.

Then, $\mathcal D^m =\mathcal D^4$ and $g_m = g_4$ for $1\leq m\leq 4$. The bubbling diagram $\mathcal D^0 = \mathcal D$ is obtained from $\mathcal D^1$ by applying bubbling moves to $\{(2,4), (2,5)\} \in D^1\setminus F^1$ and then K-bubbling moves to $\{(2,1), (3,4)\} \in D^1\setminus F^1$. We have
\[
R_1 = \{(2,2),(3,2),(2,3),(2,4)\}.
\]
The map $\psi_1\colon \{1,2,3,4,5,6\}\to\{1,2,3,4,5,6\}$ is a fixed bijection which sends $w^{-1}(1) + 1-1 = 1$ to $g_1^{-1}(4) = 4$, $w^{-1}(1) + 2-1 = 2$ to $g_1^{-1}(5) = 6$, and $w^{-1}(1) + 3-1 = 3$ to $g_1^{-1}(2) = 5$; suppose that $\psi_1$ cyclically permutes $(1,4)$ and $(2,6)$ and $(3,5)$. Then, $g_0 = g_1\circ\psi_1$ cyclically permutes $(1,4)$ and $(2,5,3)$.

See Figure~\ref{fig:Rmgm}.

\begin{figure}[htp]
\includegraphics[scale=0.7]{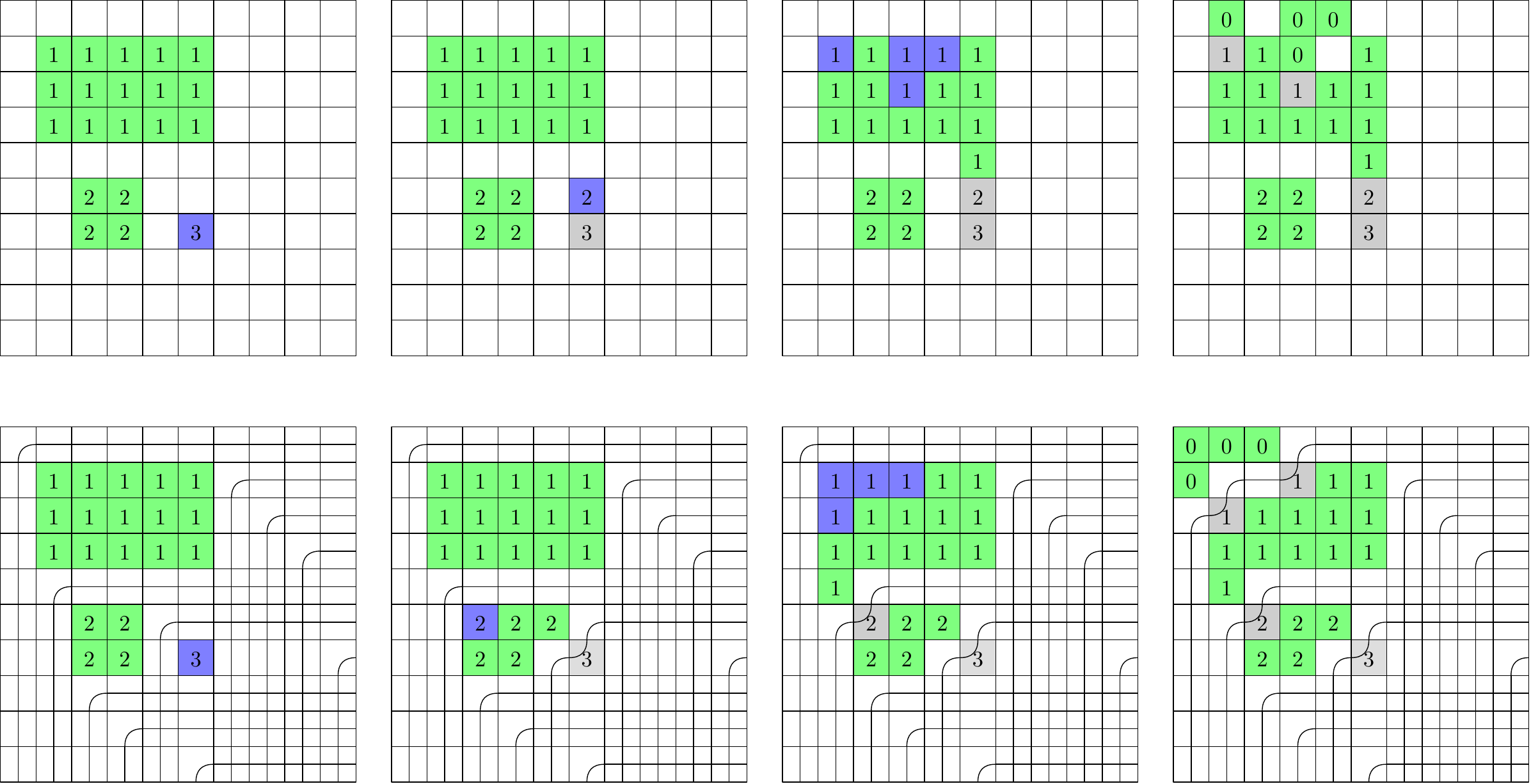}
\caption{Left column: The diagram $\mathcal D^6 = \mathcal D(w)$ for $\mathcal D$ as in Figure~\ref{fig:D-Pm-from-Dm} along with the pipe dream $(P_6, S_6) = (P(w), \emptyset)$; $R_6$ and its image under $g_6$ are shaded blue. Second column: The diagram $\mathcal D^5$ along with $(P_5, S_5)$; $R_5$ and its image under $g_5$ are shaded blue. Third column: The diagram $\mathcal D^1$ along with $(P_1, S_1)$; $R_1$ and its image under $g_1$ are shaded blue. Right column: The diagram $\mathcal D^0 = \mathcal D$ along with $(P_0, S_0) = (P,S)$.}
\label{fig:Rmgm}
\end{figure}
\end{example}

\section{Supports of top degree components of Grothendieck polynomials}
\label{sec:supports-of-top}
We prove Theorem~\ref{thm:main2} by showing that bubbling diagrams can be systematically padded to obtain a top-degree diagram which is necessarily in $\supp(\chi_{D^\top(w)})$ (Theorem~\ref{thm:top-diagrams}), and we show that divisibility relations among monomials in $\mathfrak G_w$ can be realized by inclusion relations among bubbling diagrams in a strong sense (Theorem~\ref{thm:remove-dead-squares}).
\begin{defn}
Let $w\in S_n$ be a vexillary permutation. We will construct an ordered set $A(w)$ of \textbf{distinguished live squares} using the following procedure.
\begin{enumerate}
\item Endow the squares in $D(w)$ with the total ordering given by $(i,j) \prec (i',j')$ if:
\begin{enumerate}
\item $i > i'$, or

\item $i = i'$ and $(i,j)$ has fewer squares below it than $(i',j')$, or

\item $i = i'$, $(i,j)$ has the same number of squares below it as does $(i',j')$, and $j < j'$.
\end{enumerate}
\item Add the first square in this ordering to $A(w)$.

\item Each subsequent square in the ordering will be appended to $A(w)$ if and only if $A(w)$ does not already contain a square in the same column and $A(w)$ does not already contain a square in the same linking class.
\end{enumerate}
\end{defn}

\begin{example}
The diagram $D(w)$ for $w = 178925(10)346$ with squares labeled by their positions in the $\prec$ order is shown in Figure~\ref{fig:A(w)178925(10)346}. The squares in $A(w)$ are colored gold.
\begin{figure}[ht]
\includegraphics{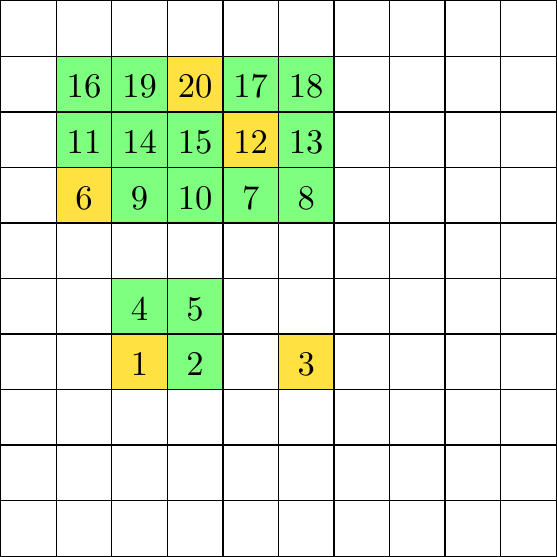}
\caption{The $\prec$ ordering on $D(w)$, along with the set $A(w)$, for $w = 178925(10)346$.}
\label{fig:A(w)178925(10)346}
\end{figure}
\end{example}
\begin{defn}
Let $w\in S_n$ be a vexillary permutation, and let $\mathcal D\in\BD(w)$. Suppose that the $k$-th square in $A(w)$ is the $m_k$-th highest square in column $j_k$. Define the \textbf{distinguished live squares} $A(\mathcal D)$ of $\mathcal D$ to be the ordered set whose $k$-th element is the $m_k$-th highest live square in $D_{j_k}$. 
\end{defn}

\begin{lem}
\label{lem:dead-squares-linked-below-Aw}
Let $w\in S_n$ be a vexillary permutation, and let $\mathcal D = (D,r,F)\in\BD(w)$. Suppose that a dead square $(i,j)\in F$ is linked to the $k$-th highest live square in $D_j$. Suppose that the $\ell$-th square in $A(\mathcal D)$ is the $k'$-th highest live square in $D_j$ for $k' < k$. Then $(i,j)$ is linked to the $\ell'$-th square in $A(\mathcal D)$ for some $\ell' < \ell$.
\end{lem}
\begin{proof}
As the $k'$-th highest live square in $D_j$ is in $A(\mathcal D)$, the $k'$-th highest square in $D(w)_j$ is in $A(w)$. Similarly, the $k$-th highest square in $D(w)_j$ is not in $A(w)$. As the $k$-th highest square in $D(w)_j$ precedes the $k'$-th highest square in the $\prec$ order, it follows that the $k$-th highest square in $D(w)_j$ is linked to the $\ell'$-th square in $A(w)$ for $\ell' < \ell$. Thus, the $k$-th highest live square in $D_j$ is linked to the $\ell'$-th square in $A(\mathcal D)$.
\end{proof}
\begin{defn}
Let $w \in S_n$ be a vexillary permutation. Construct the bubbling diagram 
\[
\mathcal D^\top(w) \colonequals (D^\top(w), r_{D(w)}^\top, F^\top(w))\in\BD(w)
\]
from $\mathcal D(w)$ by repeatedly applying bubbling moves to every square that is above a distinguished live square until it is no longer possible to do so and then repeatedly applying K-bubbling moves to every distinguished live square until it is no longer possible to do so.

For example, the bubbling diagram $\mathcal D^\top(w)$ for $w = 178925(10)346$ is shown in Figure~\ref{fig:dtop178925(10)346}.
\begin{figure}[ht]
\includegraphics{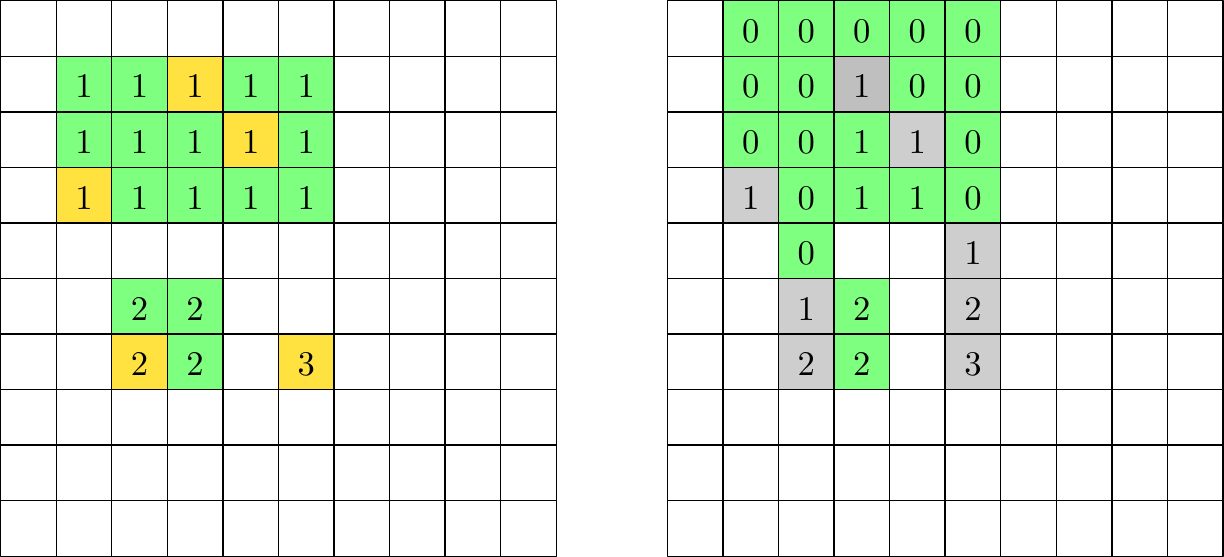}
\caption{Left: The squares in $D(w)$ are labeled by their rank, and the set $A(w)$ is colored gold. Right: The bubbling diagram $\mathcal D^\top(w)$.}
\label{fig:dtop178925(10)346}
\end{figure}
\end{defn}

\begin{thm}
\label{thm:top-diagrams}
Let $w \in S_n$ be a vexillary permutation. For any $\mathcal D = (D, r, F) \in \BD(w)$, there is $\mathcal D' = (D',r',F')\in\BD(w)$ with $F' = F^\top(w)$ and $x^{\wt(D)} \mid x^{\wt(D')}$.
\end{thm}
The proof of Theorem~\ref{thm:top-diagrams} uses the following two lemmas.
\begin{lem}
\label{lem:vexillary-diagram-semi-containment-1}
Let $w\in S_n$ be vexillary. Let $(i_1,j_1), (i_2,j_2)\in D(w)$ be two linked squares, and suppose that $i_1<i_2$. Then $(i',j_2)\in D(w)$ implies $(i',j_1)\in D(w)$ for all $i'>i_2$. In particular, $(i',j_1)$ has at least as many squares below it as $(i',j_2)$ does, and the $k$-th square of $D(w)_{j_2}$ below row $i'$ is in the same row or in a lower row than the $k$-th square of $D(w)_{j_1}$ below row $i'$.
\end{lem}
\begin{proof}
Lemma~\ref{lem:vexillary-diagram-column-copies} implies that $j_1 < j_2$. Observe that if
\begin{align*}
|\{(k,w(k))\colon k < i_2 \textup{ and } w(k)<j_2\}| - |\{(k,w(k))\colon k<i_1 \textup{ and } w(k)<j_1\}| = i_2-i_1,
\end{align*}
then there exists $(k,w(k)) \in\{(k,w(k))\colon k < i_2 \textup{ and } w(k)<j_2\}$ with $k\leq i_1$ but $w(k)>j_1$. Since $w(i_1) = j_1$, it follows that $k<i_1$.

If there exists $i' > i_2$ with $(i',j_2)\in D(w)$ and $(i',j_1)\not\in D(w)$, then $k<w^{-1}(j_1)<i'<w^{-1}(j_2)$ is a $2143$ pattern.
\end{proof}
\begin{lem}
\label{lem:vexillary-diagram-semi-containment-2}
Let $w \in S_n$ be vexillary. Let $(i,j_1), (i,j_2)\in D(w)$, and suppose that $(i,j_1)$ has more squares below it than $(i,j_2)$ does. Then $(i',j_2) \in D(w)$ implies $(i',j_1)\in D(w)$ for all $i'>i$. In particular, the $k$-th square of $D(w)_{j_2}$ below row $i$ is in the same row or in a lower row than the $k$-th square of $D(w)_{j_1}$ below row $i$.
\end{lem}
\begin{proof}
Suppose there exists $i' > i$ such that $(i',j_2)\in D(w)$ but $(i',j_1)\not\in D(w)$. Because $(i,j_1)$ has more squares below it than $(i,j_2)$ does, there exists $\ell>i$ with $(\ell,j_1)\in D(w)$ and $(\ell,j_2)\not\in D(w)$. If $j_1 < j_2$, then $j_1 < w^{-1}(\ell) < j_2 < w^{-1}(i')$ is a $2143$ pattern; if $j_2 < j_1$, then $j_2 < w^{-1}(i') < j_1 < w^{-1}(\ell)$ is a $2143$ pattern.
\end{proof}
\begin{lem}
\label{lem:vexillary-diagram-semi-containment-3}
Let $w\in S_n$ be vexillary. Let $(i_1, j_1), (i_2, j_2) \in D(w)$ be two linked squares with $j_1 < j_2$, and let $i_1'<i_1$ be minimal so that $(i,j_1) \in D(w)$ for all $i_1' \leq i \leq i_1$. Then there are exactly $i_1 - i_1'$ indices $k_1 < \dots < k_{i_1-i_1'}$ for which $i_1' \leq k_m < i_2$ and $(k_m,j_2)\in D(w)$. Furthermore, $(k_m,j_2)$ is linked to $(i_1' + m - 1, j_1)$.
\end{lem}
\begin{proof}
Lemma~\ref{lem:vexillary-diagram-column-copies} implies that $i_1 < i_2$. Thus,
\[
\{(k,w(k))\colon k<i_1 \textup{ and } w(k)<j_1\} \subseteq\{(k,w(k))\colon k < i_2 \textup{ and } w(k)<j_2\}.
\]
Since $(i,j_1)\in D(w)$ for all $i_1' \leq i \leq i_1$, we know $w(i) > j_1$ for all $i_1' \leq i \leq i_1$ and hence,
\[
\{(k,w(k))\colon k<i_1' \textup{ and } w(k)<j_1\} = \{(k,w(k))\colon k<i_1 \textup{ and } w(k)<j_1\}.
\]
If some $(k,w(k)) \in \{(k,w(k))\colon k < i_2 \textup{ and } w(k)<j_2\} \setminus \{(k,w(k))\colon k<i_1' \textup{ and } w(k)<j_1\}$ satisfies $k < i_1'$, then $w(k) \geq j_1 > w^{-1}(i_1'-1)$, and it follows that $w^{-1}(i_1'-1)<w(k)<j_2<w^{-1}(i_2)$ forms a $2143$ pattern. Because $w$ is vexillary, it follows that all $i_2 - i_1$ such elements $(k,w(k))$ satisfies $k\geq i_1'$. There are, thus, exactly $i_2 - i_1$ elements $b$ such that $i_1' \leq b < i_2$ and $(b,j_2)\not\in D(w)$; therefore, there are exactly $(i_2 - i_1') - (i_2 - i_1) = i_1 - i_1'$ elements $k$ such that $i_1' \leq k < i_2$ and $(k,j_2)\in D(w)$.

The linkedness result follows from the facts that
\[
r_{D(w)}(i,j_1) = r_{D(w)}(i_1,j_1) \qquad\textup{ for all } i_1'\leq i<i
\]
and
\[
r_{D(w)}(k_m,j_2) = r_{D(w)}(i_2,j_2) - i_2 + k_m + (i_1 - i_1' - m + 1)\qquad\textup{ for all } m\in[i_1-i_1'].\qedhere
\]
\end{proof}
\begin{proof}[Proof of Theorem~\ref{thm:top-diagrams}]
Denote the $k$-th square in $A(w)$ by $(i_k^*, j_k)$, and suppose that $(i_k^*,j_k)$ is the $m_k$-th highest square of $D(w)$ in its column. We will construct diagrams $\mathcal D^k = (D^k,r^k,F^k)$ satisfying the following properties for all $k$:

\begin{itemize}
\item $\mathcal D^k \in\BD(w)$,

\item $x^{\wt(\mathcal D^{k-1})} \mid x^{\wt(\mathcal D^k)}$,

\item $(i,j_\ell)\in D^k\setminus F^k$ for all $\ell\in[k]$ and $1\leq i \leq m_\ell$,

\item $(i,j_\ell)\in F^k$ for all $\ell\in[k]$ and $m_\ell+1\leq i \leq i_\ell^*$,

\item $(i,j_\ell)\not\in F^k$ for all $\ell\in[k]$ and $i_\ell^*+1\leq i\leq n$,

\item $\mathcal D^k_{j_\ell} = \mathcal D^{k-1}_{j_\ell}$ for all $\ell\in[k-1]$. 
\end{itemize}

Set $\mathcal D^0\colonequals\mathcal D$. Given $\mathcal D^{k-1}$, we construct $\mathcal D^k$ according to the following procedure.

Let $i_k\in D^{k-1}_{j_k}\setminus F^{k-1}_{j_k}$ be the $m_k$-th highest live square in its column. Observe that $\mathcal D^{k-1}_{j_k}$ contains no dead squares $(f,j_k)$ linked to a live square below row $i_k$: Lemma~\ref{lem:dead-squares-linked-below-Aw} guarantees that $(f,j_k)$ is linked to $(m_\ell,j_\ell)$ for some $\ell<k$, and the definition of $\prec$ guarantees that $(f,j_k)$ is linked to a dead square $(f,j_\ell)$ in the same row. In particular, $\mathcal D^{k-1}_{j_k}$ contains no dead squares below row $i_k^*+1$.

Let $L$ denote the set of dead squares in $F^{k-1}$ which are below row $i_k$ and are linked to $(i_k,j_k)$. Let $C = \{c_1, \dots, c_r\}$ be the set of columns that have a square in $L$; note that $C\cap\{j_1,\dots,j_{k-1}\}$ is empty as all dead cells in column $j_\ell$ are linked to $(m_\ell,j_\ell)$ and hence not linked to $(i_k,j_k)$.

We shall first move squares in $L$ horizontally between the columns in $C$ and reindex the $c_i$ so that whenever $k < k'$, all squares of $L$ in column $c_k$ are above all squares of $L$ in column $c_{k'}$ using the following process. Let $(r_i^-, c_i)$ denote the live square in column $c_i$ that is linked to $(i_k,j_k)$, and let $r_i^+>r_i^-$ be minimal such that $(r_i^+,c_i)$ is live. The squares in $(i,c_j)$ for $r_j^-< i < r_j^+$ are either dead or empty; if they are dead, then they are linked to $(i_k,j_k)$. Furthermore, if $r_j^- < i < r_j^+$ and $r_k^- < i < r_k^+$, then $(i,c_j)$ and $(i,c_k)$ cannot both be dead. If $[r_i^-,r_i^+]\supseteq[r_j^-,r_j^+]$ for some $i,j$, then we may move all dead squares in $\{(i,c_j)\colon r_j^-<i<r_j^+\}$ to column $c_i$ (to break a tie $[r_i^-,r_i^+]=[r_j^-,r_j^+]$, we move all dead squares to the column with the smaller index). Now, reorder the $c_i$ so that $r_1^- > r_2^- > \dots$, and move all dead squares in $\{(i,c_j)\colon r_j^-<i<r_j^+\}$ to $(i,c_k)$ for $k$ minimal such that $i<r_k^+$.

 We now modify each column $c_i \in C$, starting from $c_1$ and working towards $c_r$, according to the following procedure:

Let $x_1 < \dots < x_{k_1}$ be the rows below $i_k$ where $(x_s, c_i)$ is dead and linked to $(i_k,j_k)$ and where $(x_s,j_k)$ is live.  Also, let $y_1 < \dots < y_{k_2}$ be the rows below $i_k$ where $(y_s,c_i)$ is dead and linked to $(i_k,j_k)$ and where $(y_s,j_k)$ is empty. Let $z_1 < \dots < z_{k_1}$ be the first $k_1$ rows below row $\max\{x_{k_1},y_{k_2}\}$ where $(z_s, c_i)$ is live and $(z_s,j_k)$ is empty; such rows $z_1, \dots, z_{k_1}$ exist because the live square $(a,c_i)$ immediately above the dead squares $(x_s,c_i)$ and $(y_s,c_i)$ is linked to $(i_k,j_k)\in A(\mathcal D^k)$, so Lemma~\ref{lem:vexillary-diagram-semi-containment-1} implies $(a,c_i)$ has at least as many live squares below it as as $(i_k,j_k)$ does.

Let $k_3$ be the number of rows between $i_k$ and $\max\{x_{k_1}, y_{k_2}\}$ which have an empty space in column $c_i$ and a live square in column $j_k$. Modify the portions of columns $c_i$ and $j_k$ below row $i_k$ such that:
\begin{itemize}
\item Column $c_i$ has live squares in rows $x_1, \dots, x_{k_1}$ and any rows below $i_k$ which previously had live squares, except for rows $z_1, \dots, z_{k_1}$,

\item Column $j_k$ has dead squares in all rows between $i_k + 1$ and $\max\{x_{k_1}, y_{k_2}\}$, inclusive, along with dead squares in any other rows which already has dead squares, and

\item Column $j_k$ has live squares in all rows below $\max\{x_{k_1}, y_{k_2}\}$ which already had a live square, rows $z_1, \dots, z_{k_1}$, and the first other $k_3$ rows below $\max\{x_{k_1}, y_{k_2}\}$.
\end{itemize}

See Figure~\ref{fig:dtop-procedure1} for an example. Letting $S_{c_i; i}$ and $S_{j_k;i}$ denote the set of live squares in the modified columns $c_i$ and $j_k$ respectively, Lemmas~\ref{lem:vexillary-diagram-semi-containment-1} and~\ref{lem:vexillary-diagram-semi-containment-2} imply that $S_{c_i;i} \leq D(w)_{c_i}$ and $S_{j_k;i}\leq D(w)_{j_k}$. Then, Lemma~\ref{lem:BD-characterization} implies that the resulting diagram is in $\BD(w)$.

\begin{figure}[ht]
\includegraphics{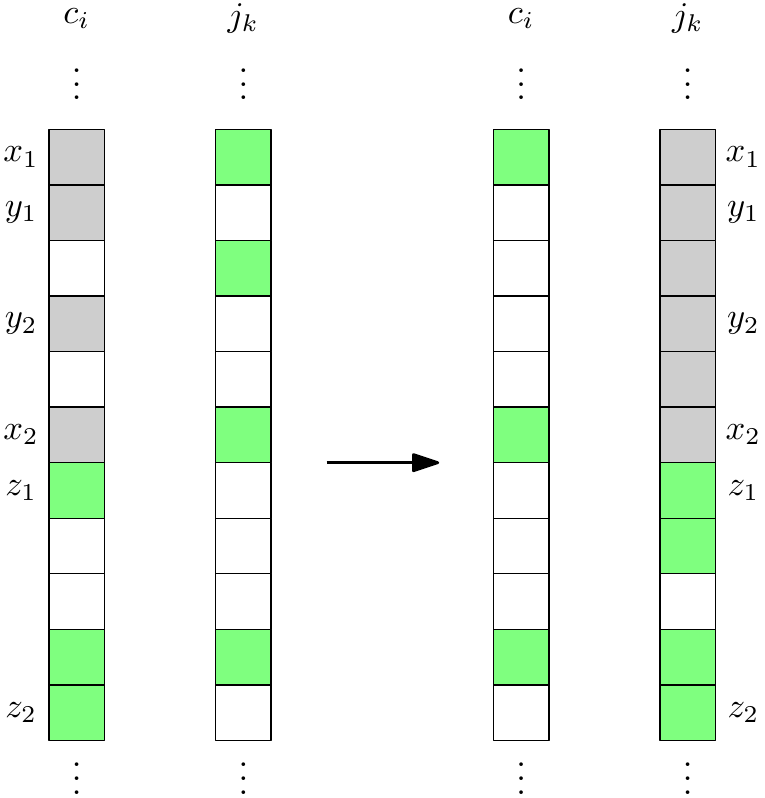}
\caption{In this example, $k_3 = 1$ because row $(y_1+1,c_i)$ is empty and $(y_1+1,j_k)$ is live. After performing the procedure, $(z_2-3,j_k)$ is live because it is the first row below $\max\{x_2,y_2\}$ which did not already have a live square and is not row $z_1$ or $z_2$.}
\label{fig:dtop-procedure1}
\end{figure} 

At this point, every square in $L$ is in column $j_k$. We now use the following procedure to bubble up squares in column $j_k$ so that $(i,j_k)$ is live for all $\ell\in[k]$ and $1\leq i\leq m_\ell$:
\begin{itemize}
\item If there is a column $j\neq j_k$ such that $(i_k,j)\in F$ is dead and linked to $(i_k,j_k)$, swap the portions of columns $j$ and $j_k$ in and above row $i_k$. Then fill in dead squares between $(i_k,j_k)$ and the next lowest live square above it, removing matching dead squares from other columns if necessary. Such a $j$ is necessarily not equal to $j_\ell$ for $\ell\in[k-1]$, as those columns contain only dead squares linked to $(m_\ell,j_\ell)$. Letting $S_j$ and $S_{j_k}$ denote the set of live squares in the modified columns $j$ and $j_k$ respectively, Lemmas~\ref{lem:vexillary-diagram-column-copies} and~\ref{lem:vexillary-diagram-semi-containment-3} imply that $S_j\leq D(w)_j$ and $S_{j_k}\leq D(w)_{j_k}$. Then Lemma~\ref{lem:BD-characterization} implies that the resulting diagram is in $\BD(w)$.

\item If there is no such column, and for the maximal $i < i_k$ so that $i\not\in D_{j_k}\setminus F_{j_k}$, $(i,j_k)$ is empty, then apply bubbling moves at $i+1,\dots,i_k-1$ followed by a K-bubbling move at $i_k$. The resulting diagram is in $\BD(w)$.

\item If there is no such column, and for the maximal $i < i_k$ so that $i\not\in D_{j_k}\setminus F_{j_k}$, $(i,j_k)$ is dead, then remove this dead square and apply bubbling moves at $i+1,\dots,i_k-1$ followed by a K-bubbling move at $i_k$. The resulting diagram is in $\BD(w)$.
\end{itemize}

When this procedure terminates, the square $(i,j_k)$ is live for all $1\leq i \leq m_k$, is dead for all $m_k+1\leq i\leq i_k^*$, and is live or empty otherwise. Columns $j_1, \dots, j_{k-1}$ were left invariant throughout this construction. We may push down any remaining live squares in $\{(i_k+1,j_k),\dots,(i_k^*,j_k)\}$ so that there are no live squares in this region and then fill in any empty squares in $\{(i_k+1,j_k),\dots,(i_k^*,j_k)\}$ with dead squares. We set $\mathcal D^k$ to be the resulting diagram.
\end{proof}

\newtheorem*{thm:main2}{Theorem~\ref{thm:main2}}
\begin{thm:main2}
Let $w \in S_n$ be a vexillary permutation. Then $\supp(\mathfrak G_w^\top) = \supp(\chi_{D^\top(w)})$.\end{thm:main2}
\begin{proof}[Proof of Theorem~\ref{thm:main2}]
Theorem~\ref{thm:top-diagrams} implies that any monomial appearing in $\mathfrak G_w^\top$ is equal to $x^{\wt(D)}$ for some $(D,r,F)\in\BD(w)$ with $F = F^\top(w)$. Any such diagram $D$ satisfies $D\leq D^\top(w)$, so
\[
\supp(\mathfrak G_w^\top) \subseteq \{\wt(D)\colon D\leq D^\top(w)\}.
\]
By construction, $D^\top(w) \in \BD(w)$. Furthermore, if $(i,j)\in F^\top(w)$, then $(i',j) \in D^\top(w)$ for all $i' < i$. It follows (e.g.\ by Lemma~\ref{lem:BD-characterization}) that 
\[\{
\wt(D)\colon D\leq D^\top(w)\}\subseteq \{\wt(\mathcal D)\colon \mathcal D\in\BD(\mathcal D^\top(w))\}\subseteq\supp(\mathfrak G_w^\top).\]
We conclude that
\[
\supp(\mathfrak G_w^\top) = \{\wt(D)\colon D\leq D^\top(w)\}.
\]
Finally, Proposition~\ref{prop:support-of-chi} guarantees that $\{\wt(D)\colon D\leq D^\top(w)\} = \supp(\chi_{D^\top(w)})$.
\end{proof}

The next result asserts that if a monomial $x^{\wt(D)}$ appearing in $\mathfrak G_w$ is represented by a bubbling diagram $\mathcal D = (D,r,F)$, then any monomial $x^\alpha$ which divides $x^{\wt(D)}$ and appears in $\mathfrak G_w$ can be represented by a bubbling diagram whose dead squares are contained in $F$.
\begin{thm}
\label{thm:remove-dead-squares}
Let $\mathcal D = (D,r,F) \in \BD(w)$. Suppose that there exists $i\in[n]$ so that $x^{\wt(D)}/x_i$ appears with nonzero coefficient in $\mathfrak G_w$. Then there is $\mathcal D^- = (D^-, r^-, F^-) \in\BD(w)$ so that $\wt(D^-) = \wt(D) - e_i$, $F^-\subsetneq F$, and $r=r^-$ on $F^-$.
\end{thm}
\begin{proof}
Fix a diagram $\mathcal C = (C, s, G) \in\BD(w)$ so that $\wt(C) = \wt(D) - e_i$.

If row $i$ of $\mathcal D$ contains a dead square, then removing that square gives the desired diagram $\mathcal D^-$. Otherwise, there must be a square $(i,j)$ so that $(i,j)\in D\setminus F$ and $(i,j)\not\in C$. Suppose that $(i,j)\in D\setminus F$ is the $k$-th uppermost live square in the column $D_j$. There are two cases:
\begin{enumerate}
\item Suppose that the $k$-th uppermost live square in the column $C_j$ is above row $i$. Let $i' < i$ be maximal so that $D_j$ does not have a live square in the $i'$-th row; such a position exists because $C_j$ has its $k$-th uppermost square above row $i$. Apply a bubbling move to the live squares $(i,j), (i-1,j), \dots, (i'+1,j)$ of $D$. If $(i',j)\in F$, then simply remove it to make the bubbling move legal. The resulting diagram is in $\BD(w)$.

\item Suppose that the $k$-th uppermost live square in the column $C_j$ is below row $i$. Let $i' > i$ be minimal so that $D_j$ does not have a live square in the $i'$-th row. Because the $k$-th uppermost live square in $C_j$ is below row $i$, the diagram obtained from $D$ by ``pushing down'' the live squares $(i,j), (i+1, j), \dots, (i'-1,j)$ of $D$ by one space, removing a dead square at $(i',j)$ if it exists, is again a diagram in $\BD(w)$.
\end{enumerate}
In either case, if a dead square was removed then the resulting diagram has weight $\wt(D) - e_i$, giving our desired bubbling diagram $\mathcal D^-$. If no dead square was removed, then the resulting diagram has weight $\wt(D) - e_i + e_{i'}$ and has more squares in row $i'$ than does $C$, so we may repeat the process using row $i'$ until a dead square is removed.

At each step of the process, the squares in $D$ move closer to their counterparts in $C$. Because $\wt(C) = \wt(D) - e_i$ and $|D\setminus F| = |C\setminus G|$, there is a row $r$ so that $\#\{j\colon(r,j)\in C\setminus G\} > \#\{j\colon(r,j)\in D\setminus F\}$. In particular, there is a column $j_*$ in which a live square $(r,j_*)\in C\setminus G$ is not in the same row as its counterpart in $(r', j_*)\in D\setminus F$; the algorithm will eventually move $(r',j_*)$ to $(r,j_*)$, so this procedure will terminate.
\end{proof}
\begin{defn}
	\label{defn:sbd}
Let $\SBD(w)$ denote the set of bubbling diagrams $\mathcal D = (D,r,F)\in\mathcal{BD}(w)$ for which every dead square is linked to a distinguished live square in its column.
\end{defn}
\begin{thm}
\label{thm:supp-Gw-SBD}
If $w\in S_n$ is vexillary, then $\supp(\mathfrak G_w) = \{\wt(\mathcal D)\colon \mathcal D \in \SBD(w)\}$.
\end{thm}
\begin{proof}
%
By Theorem~\ref{thm:main1}, any monomial appearing in $\mathfrak G_w$ is of the form $x^{\wt(D)}$ for some $\mathcal D = (D,r,F)\in\BD(w)$. By Theorem~\ref{thm:top-diagrams}, there is a bubbling diagram $\mathcal C = (C, s, F^\top(w))\in\mathcal{BD}(\mathcal D^\top(w))$ so that $x^{\wt(D)}\mid x^{\wt(C)}$. Repeated application of Theorem~\ref{thm:HafSupportBetween} and Theorem~\ref{thm:remove-dead-squares} gives the desired result. 
\end{proof}
Observe that $\SBD(w)$ is precisely the set of diagrams which can be generated from $D(w)$ by any series of the following moves:
\begin{enumerate}
  \item Bubble up any live square
  \item K-bubble any distinguished live square
\end{enumerate}
In particular, once the set $A(w)$ of distinguished live squares has been determined, this procedure makes no further reference to the ranks of squares (since no pair of squares in $A(w)$ can be linked). The possible states of each column in $\SBD(w)$ are, thus, independent of the states of the other columns. Figure~\ref{fig:SBDExample} shows an example of $\SBD(w)$.
\begin{figure}
  \centering
  \begin{subfigure}{0.35\textwidth}
  \scalebox{1.2}{
  \begin{tikzpicture}
\node[scale=.75] at (-.7,1.75) {$D(w)=$};
  \filldraw[draw=black, color=green, opacity=0.5] (.5,2) rectangle (1.5,2.5);
  \filldraw[draw=black, color=green, opacity=0.5] (1,1.5) rectangle (1.5,2);
  \filldraw[draw=black, color=gold, opacity=0.75] (2,1.5) rectangle (2.5,2);
  \filldraw[draw=black, color=gold, opacity=0.75] (.5,1.5) rectangle (1,2);
  \draw[step=.5cm,black,very thin] (0,0) grid (3,3);
  \end{tikzpicture}}
  \end{subfigure}
  \hfill
  \begin{subfigure}{0.63\textwidth}
  \begin{subfigure}{0.63\textwidth}
  \scalebox{0.5}{
  \begin{tikzpicture}
  \node[scale=1.2] at (-1.5,2.75) {Column $2$:};
    \filldraw[draw=black, color=green, opacity=0.5] (0,2) rectangle (.5,2.5);
    \filldraw[draw=black, color=gold, opacity=0.75] (0,1.5) rectangle (.5,2);
    \draw[step=.5cm,black,very thin] (0,0) grid (.5,3);
    \filldraw[draw=black, color=green, opacity=0.5] (1,2.5) rectangle (1.5,3);
    \filldraw[draw=black, color=gold, opacity=0.75] (1,1.5) rectangle (1.5,2);
    \draw[step=.5cm,black,very thin] (.999,0) grid (1.5,3);
    \filldraw[draw=black, color=green, opacity=0.5] (2,2.5) rectangle (2.5,3);
    \filldraw[draw=black, color=gold, opacity=0.75] (2,2) rectangle (2.5,2.5);
    \draw[step=.5cm,black,very thin] (1.999,0) grid (2.5,3);
    \filldraw[draw=black, color=green, opacity=0.5] (3,2.5) rectangle (3.5,3);
    \filldraw[draw=black, color=gold, opacity=0.75] (3,2) rectangle (3.5,2.5);
    \filldraw[draw=black, color=lightgray] (3,1.5) rectangle (3.5,2);
    \draw[step=.5cm,black,very thin] (2.999,0) grid (3.5,3);
  \end{tikzpicture}
  \hspace{1cm}
  \begin{tikzpicture}
  \node[scale=1.2] at (-1.5,2.75) {Column $3$:};
    \filldraw[draw=black, color=green, opacity=0.5] (0,2) rectangle (.5,2.5);
    \filldraw[draw=black, color=green, opacity=0.5] (0,1.5) rectangle (.5,2);
    \draw[step=.5cm,black,very thin] (0,0) grid (.5,3);
    \filldraw[draw=black, color=green, opacity=0.5] (1,2.5) rectangle (1.5,3);
    \filldraw[draw=black, color=green, opacity=0.5] (1,1.5) rectangle (1.5,2);
    \draw[step=.5cm,black,very thin] (.999,0) grid (1.5,3);
    \filldraw[draw=black, color=green, opacity=0.5] (2,2.5) rectangle (2.5,3);
    \filldraw[draw=black, color=green, opacity=0.5] (2,2) rectangle (2.5,2.5);
    \draw[step=.5cm,black,very thin] (1.999,0) grid (2.5,3);
  \end{tikzpicture}}
  \vspace{0.1cm}
  \end{subfigure} 
  \begin{subfigure}{0.63\textwidth}
  \scalebox{0.5}{
  \begin{tikzpicture}
  \node[scale=1.2] at (-1.5,2.75) {Column $5$:};
    \filldraw[draw=black, color=gold, opacity=0.75] (0,1.5) rectangle (.5,2);
    \draw[step=.5cm,black,very thin] (0,0) grid (.5,3);
    \filldraw[draw=black, color=gold, opacity=0.75] (1,2) rectangle (1.5,2.5);
    \draw[step=.5cm,black,very thin] (.999,0) grid (1.5,3);
    \filldraw[draw=black, color=lightgray] (2,1.5) rectangle (2.5,2);
    \filldraw[draw=black, color=gold, opacity=0.75] (2,2) rectangle (2.5,2.5);
    \draw[step=.5cm,black,very thin] (1.999,0) grid (2.5,3);
    \filldraw[draw=black, color=gold, opacity=0.75] (3,2.5) rectangle (3.5,3);
    \draw[step=.5cm,black,very thin] (2.999,0) grid (3.5,3);
    \filldraw[draw=black, color=gold, opacity=0.75] (4,2.5) rectangle (4.5,3);
    \filldraw[draw=black, color=lightgray] (4,1.5) rectangle (4.5,2);
    \draw[step=.5cm,black,very thin] (3.999,0) grid (4.5,3);
    \filldraw[draw=black, color=gold, opacity=0.75] (5,2.5) rectangle (5.5,3);
    \filldraw[draw=black, color=lightgray] (5,2) rectangle (5.5,2.5);
    \draw[step=.5cm,black,very thin] (4.999,0) grid (5.5,3);
    \filldraw[draw=black, color=gold, opacity=0.75] (6,2.5) rectangle (6.5,3);
    \filldraw[draw=black, color=lightgray] (6,2) rectangle (6.5,2.5);
    \filldraw[draw=black, color=lightgray] (6,1.5) rectangle (6.5,2);
    \draw[step=.5cm,black,very thin] (5.999,0) grid (6.5,3);
  \end{tikzpicture}}
  \end{subfigure}
  \end{subfigure}
  \caption{Construction of $\SBD(w)$ for $w=146235$. The set $A(w)$ of distinguished live squares is shown in gold. Any combination of the above options for columns $2$, $3$, and $5$ will yield a valid diagram in $\SBD(w)$}
  \label{fig:SBDExample}
\end{figure}
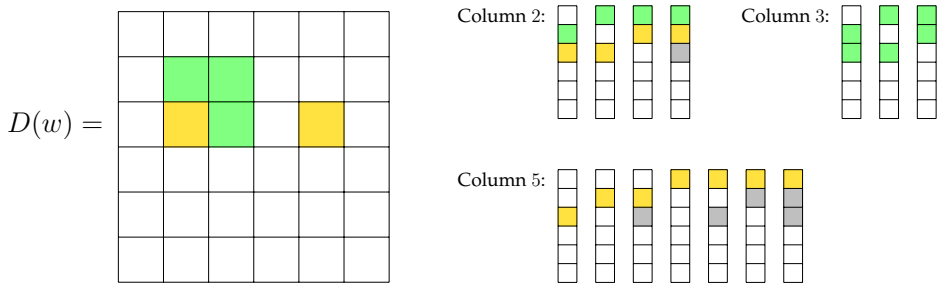
\section{Supports of homogenized Grothendieck polynomials}
\label{sec:supports-of-homog}
We deduce Theorem~\ref{thm:main3} from a ``one-column version'' of the result (Proposition~\ref{prop:one-column-M-convexity}).
\begin{defn}[{\cite{mty2019}}]
Let $D\subseteq[n]\times[k]$ be a diagram. The \textbf{Schubitope} $\mathcal S_D$ is the Newton polytope of the dual character $\chi_D$ of the flagged Weyl module.
\end{defn}
By~\cite{fms2018}, the Schubitope $\mathcal S_D$ is the Minkowski sum
\[
\mathcal S_D = \sum_{i=1}^k P(\SM_n(D_i))
\]
of Schubert matroid polytopes.

We recall the combinatorial interpretation, due to \cite{mty2019}, for the rank functions of Schubert matroids. For $I,J\subseteq[n]$, construct a string denoted $\word_I^n(J)$ by setting $k = 1, \dots, n$ and recording
\begin{itemize}
\item $\underline{\hspace{0.2cm}}$ if $k\not\in I$ and $k\not\in J$;

\item $($ if $k\not\in I$ and $k\in J$;

\item $)$ if $k\in I$ and $k\not\in J$;

\item $\star$ if $k\in I$ and $k\in J$.
\end{itemize}
Define
\[
\theta_I^n(J) \overset{\rm def}= \#\{\textup{matched $()$'s in $\word_I^n(J)$}\} + \#\{\textup{$\star$'s in $\word_I^n(J)$}\}
\]
where parentheses are matched iteratively left-to-right, removing matched pairs.

\begin{example}
\label{ex:theta_example}
	Let $n=14$, $I=\{2,4,6,9,10,11,12,13\}$, and $J=\{3,4,5,8,9,12,14\}$. Coloring $I-J$, $J-I$, and $I\cap J$ respectively red, blue, and purple, we show how to compute $\theta_I^n(J)$ in Figure~\ref{fig:theta_example}.
\end{example}

\begin{figure}[ht]
	\begin{center}
		\includegraphics[scale=0.8]{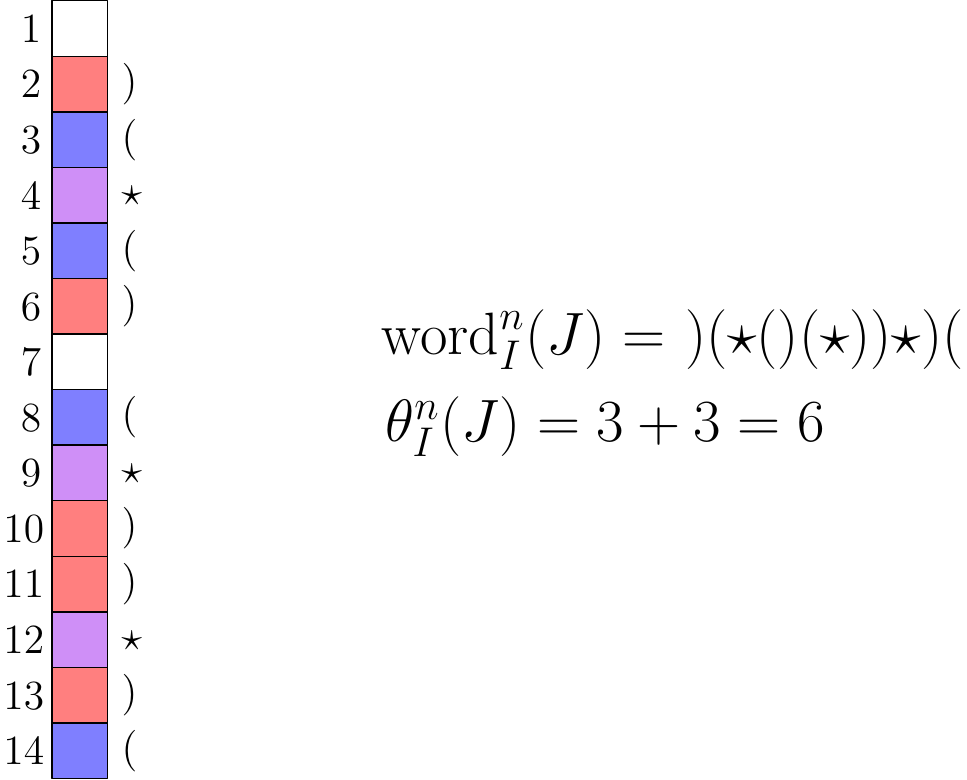}
	\end{center}
	\caption{Left: the sets $I$, $J$, and $I\cap J$ from Example~\ref{ex:theta_example}. Right: the word $\word_I^n(J)$ and the number $\theta_I^n(J)$.}
	\label{fig:theta_example}
\end{figure}

\begin{thm}[{\cite{fms2018}*{Theorem 10}}]
For any $n$ and any $I,J\subseteq[n]$,
\[
r_{\SM_n(I)}(J) = \theta_I^n(J).
\]
\end{thm}

Let $S\subseteq[n]$ be a one-column diagram with a single distinguished square $s\in S$. Set $S^{(0,s)}\colonequals S$, and whenever $[s]\not\subseteq S^{(k-1,s)}$, define $S^{(k,s)}$ from $S^{(k-1,s)}$ by 
\[
S^{(k,s)} = S^{(k-1,s)}\cup \max\{i\colon i < s \text{ and } i \not\in S^{(k-1,s)}\}.
\]
Let $d = s - |\{i\in S\colon i \leq s\}|$ so that $[s]\subseteq S^{(d,s)}$.
\begin{lem}
\label{lem:rank-functions-explicit}
For all $k$, we have 
\[
r_{\SM_n(S^{(k+1,s)})}(J) - r_{\SM_n(S^{(k,s)})}(J) \in \{0,1\}.
\]
When $r_{\SM_n(S^{(k+1,s)})}(J) = r_{\SM_n(S^{(k,s)})}(J)$, then we have 
\begin{align*}
r_{\SM_n(S^{(k'+1,s)})}(J) &= r_{\SM_n(S^{(k',s)})}(J) \qquad \textup{ for all } k' > k,\textup{ and } \\
r_{\SM_n(S^{(k+1,s)})}(J') &= r_{\SM_n(S^{(k,s)})}(J') \qquad \textup{ for all } J' \subset J.
\end{align*}
\end{lem}
\begin{proof}
Suppose that $S^{(k+1,s)}\setminus S^{(k,s)} = \{i\}$. If $i \in J$, then $\word_{S^{(k+1,s)}}^n(J)$ is obtained from $\word_{S^{(k,s)}}^n(J)$ by replacing the $($ in the $i$-th position of $\word_{S^{(k,s)}}^n(J)$ with a $\star$, while if $i\not \in J$, then $\word_{S^{(k+1,s)}}^n(J)$ is obtained from $\word_{S^{(k,s)}}^n(J)$ by replacing the $\underline{\hspace{0.2cm}}$ in the $i$-th position of $\word_{S^{(k,s)}}^n(J)$ with a $)$. In either case, $r_{\SM_n(S^{(k+1,s)})}(J) - r_{\SM_n(S^{(k,s)})}(J) \in \{0,1\}$.

Now, fix $k' > k$, and write $S^{(k'+1,s)}\setminus S^{(k',s)} = \{i'\}$. Note that $i' < i$.

Suppose that $r_{\SM_n(S^{(k+1,s)})}(J) = r_{\SM_n(S^{(k,s)})}(J)$. By considering separately the cases $i \in J$ and $i\not\in J$, one can deduce that every $($ in the $\ell$-th position of $\word_{S^{(k,s)}}^n(J)$, $\ell \leq i$, is matched to a $)$. Thus:
\begin{itemize}
\item If $i' \in J$, then replacing a $($ with a $\star$ decreases the number of matched $()$'s by one while increasing the number of $\star$'s by one. Thus, $r_{\SM_n(S^{(k'+1,s)})}(J) = r_{\SM_n(S^{(k',s)})}(J)$. 

\item If $i'\not\in J$, then replacing the $\underline{\hspace{0.2cm}}$ with a $)$ in the $i'$-th position of $\word_{S^{(k',s)}}^n(J)$ does not increase the number of matched $()$'s, as every $($ to the left of position $i'$ was already matched. Thus, $r_{\SM_n(S^{(k'+1,s)})}(J) = r_{\SM_n(S^{(k',s)})}(J)$.

\end{itemize}
Now fix $J'\subset J$. Note that for any set $T$, $\word_T^n(J')$ is obtained from $\word_T^n(J)$ by replacing, for every $j\in (J\setminus J')\cap T$, the $\star$ in the $j$-th position with a $)$ and, for every $j\in (J\setminus J')\setminus T$, the $($ in the $j$-th position with a $\underline{\hspace{0.2cm}}$.

Suppose that $r_{\SM_n(S^{(k+1,s)})}(J) = r_{\SM_n(S^{(k,s)})}(J)$. We know that every $($ in the $\ell$-th position of $\word_{S^{(k,s)}}^n(J)$, $\ell \leq i$, is matched to a $)$. Because $\word_T^n(J')$ is obtained from $\word_T^n(J)$ by replacing $\star$'s with $)$'s and $($'s with $\underline{\hspace{0.2cm}}$'s, every $($ in the $\ell$-th position of $\word_{S^{(k,s)}}^n(J')$, $\ell \leq i$, is matched to a $)$. Thus:
\begin{itemize}
\item If $i \in J'$, then replacing a $($ with a $\star$ decreases the number of matched $()$'s by one while increasing the number of $\star$'s by one. Thus, $r_{\SM_n(S^{(k+1,s)})}(J') = r_{\SM_n(S^{(k,s)})}(J')$. 

\item If $i\not\in J'$, then replacing the $\underline{\hspace{0.2cm}}$ in the $i$-th position of $\word_{S^{(k,s)}}^n(J')$ with a $)$ does not change the number of matched $()$'s, as every $($ to the left of position $i$ was already matched. Thus, $r_{\SM_n(S^{(k+1,s)})}(J') = r_{\SM_n(S^{(k,s)})}(J')$. 
\end{itemize}
\end{proof}
\begin{cor}
\label{cor:rank-functions-explicit}
The Schubert matroid rank function of $S^{(k,s)}$ is given by
\[
r_{\SM_n(S^{(k,s)})}(I) = \min\{r_{\SM_n(S^{(d,s)})}(I), r_{\SM_n(S)}(I)+k\} \qquad\textup{ for all } I\subseteq[n].
\]
Furthermore, if $J'\subseteq J$, then
\[
r_{\SM_n(S^{(d,s)})}(J) - r_{\SM_n(S^{(d,s)})}(J') \geq r_{\SM_n(S)}(J) - r_{\SM_n(S)}(J').
\]
\end{cor}
\begin{proof}
By Lemma~\ref{lem:rank-functions-explicit}, we know $r_{\SM_n(S^{(k+1,s)})}(J) - r_{\SM_n(S^{(k,s)})}(J) \in \{0,1\}$. This implies
\[
r_{\SM_n(S^{(k,s)})}(J)\leq\min\{r_{\SM_n(S^{(d,s)})}(I), r_{\SM_n(S)}(I)+k\}.
\]
Furthermore, Lemma~\ref{lem:rank-functions-explicit} implies that:
\begin{itemize}
\item If $r_{\SM_n(S^{(k+1,s)})}(J) - r_{\SM_n(S^{(k,s)})}(J) = 0$, then $r_{\SM_n(S^{(k'+1,s)})}(J) - r_{\SM_n(S^{(k',s)})}(J) = 0$ for all $k' > k$, so $r_{\SM_n(S^{(k,s)})}(J) = r_{\SM_n(S^{(d,s)})}(J)$.

\item If $r_{\SM_n(S^{(k+1,s)})}(J) - r_{\SM_n(S^{(k,s)})}(J) = 1$, then $r_{\SM_n(S^{(k'+1,s)})}(J) - r_{\SM_n(S^{(k',s)})}(J) = 1$ for all $k' < k$, so $r_{\SM_n(S^{(k,s)})}(J) = r_{\SM_n(S)}(J) + k$.
\end{itemize}
If $J'\subseteq J$, then Lemma~\ref{lem:rank-functions-explicit} implies that
\[
r_{\SM_n(S^{(d,s)})}(J) - r_{\SM_n(S)}(J) \geq r_{\SM_n(S^{(d,s)})}(J') - r_{\SM_n(S)}(J').
\]
Rearranging the inequality gives the desired result.
\end{proof}
For $B \leq S^{(k,s)}$, let $\widetilde \zeta^B = (\zeta^B_1, \dots, \zeta^B_n,d-k)\in\RR^{n+1}$ be the vector with $\zeta^B_i = 1$ if $i \in B$ and $\zeta^B_i = 0$ if $i\not\in B$. Define the polytope
\[
\mathcal P(S^{(s)}) \overset{\rm def}=\mathrm{conv}\{\widetilde\zeta^B\colon B\leq S^{(k,s)}\textup{ for some } k\leq d\}.
\]
\begin{prop}
\label{prop:one-column-M-convexity}
The polytope $\mathcal P(S^{(s)})$ is a generalized permutahedron, and 
\[
\mathcal P(S^{(s)})\cap\ZZ^{n+1} = \{\widetilde\zeta^B\colon B\leq S^{(k,s)}\textup{ for some } k\leq d\}.
\]
\end{prop}
\begin{proof}
Consider the function $z\colon2^{[n+1]}\to\RR$ defined by
\[
z(I) = \begin{cases} r_{\SM_n(S^{(d,s)})}(I) &\textup{ if } n+1\not\in I\\
r_{\SM_n(S)}(I\setminus\{n+1\})+d &\textup{ if } n+1\in I.\end{cases}
\]
We claim that $z$ is submodular. Indeed:
\begin{itemize}
\item If $n+1\not\in I,J$, then $z(I) + z(J)\geq z(I\cup J) + z(I\cap J)$ because $r_{\SM_n(S^{(d,s)})}$ is submodular,

\item If $n+1 \in I\setminus J$, then
\begin{align*}
z(I) + z(J) &= r_{\SM_n(S)}(I\setminus \{n+1\}) + d + r_{\SM_n(S^{(d,s)})}(J) \\
&\geq r_{\SM_n(S)}(I\setminus\{n+1\}) + d + r_{\SM_n(S)}(J) - r_{\SM_n(S)}(I\cap J) + r_{\SM_n(S^{(d,s)})}(I\cap J)\\
&\geq d + r_{\SM_n(S)}(I\cup J\setminus\{n+1\}) + r_{\SM_n(S^{(d,s)})}(I\cap J)\\
&=z(I\cup J) + z(I\cap J),
\end{align*}
where the first inequality uses Corollary~\ref{cor:rank-functions-explicit} applied to $r_{\SM_n(S^{(d,s)})}(J)$ and the second inequality uses the submodular inequality $r_{\SM_n(S)}(J) - r_{\SM_n(S)}(I\cap J) \geq r_{\SM_n(S)}(I\cup J\setminus\{n+1\}) - r_{\SM_n(S)}(I\setminus\{n+1\})$.

\item If $n+1 \in I,J$, then $z(I) + z(J) \geq z(I\cup J) + z(I\cap J)$ because $r_{\SM_n(S)}$ is submodular.
\end{itemize}
Since $z$ is submodular, we have a generalized permutahedron
\[
P = \left\{t\in\RR^{n+1}\colon\sum_{i\in I}t_i\leq z(I) \textup{ for all } I\subseteq[n+1] \textup{ and } \sum_{i=1}^{n+1}t_i = z([n+1])\right\}.
\]
We now claim that $\mathcal P(S^{(s)}) = P$. To prove this, fix any $B\leq S^{(k,s)}$ and $I\subseteq[n+1]$. If $n+1\not\in I$, then
\[
\sum_{i\in I}\zeta^B_i \leq r_{\SM_n(S^{(k,s)})}(I) \leq r_{\SM_n(S^{(d,s)})}(I),
\]
and if $n+1\in I$, then
\begin{align*}
\left(\sum_{i\in I\setminus\{n+1\}}\zeta^B_i\right) + d-k &\leq r_{\SM_n(S^{(k,s)})}(I\setminus\{n+1\}) + d-k \\&\leq r_{\SM_n(S)}(I\setminus\{n+1\}) + d,
\end{align*}
where we use the inequality $r_{\SM_n(S^{(k,s)})}(I) - k \leq r_{\SM_n(S)}(I)$ from Corollary~\ref{cor:rank-functions-explicit}. Furthermore,
\[
\sum_{i\in[n]}\zeta^B_i + d-k = (|S| + k) + d-k = z([n+1]),
\]
so $\widetilde\zeta^B \in P$. We conclude that $\mathcal P(S^{(s)})\subseteq P$.

Now fix any $t \in P\cap\ZZ^{n+1}$. Observe that $z([n]) = z([n+1])$, so $t_{n+1} \geq 0$. Furthermore, $z(\{n+1\}) = d$, so $t_{n+1}\leq d$. Write $t_{n+1} = d - k$. Observe that for any $I\subseteq[n]$, we have
\[
\sum_{i\in I}t_i \leq z(I) = r_{\SM_n(S^{(d,s)})}(I) \qquad\textup{ and } \qquad \sum_{i\in I}t_i \leq z(I\cup\{n+1\}) - t_{n+1} = r_{\SM_n(S)}(I\setminus\{n+1\}) + k,
\]
so that
\[
\sum_{i\in I}t_i\leq \min\{r_{\SM_n(S^{(d,s)})}(I), r_{\SM_n(S)}(I\setminus\{n+1\}) + k\} = r_{\SM_n(S^{(k,s)})}(I).
\]
In particular, $(t_1, \dots, t_n)$ is an integer point of the Schubitope $\mathcal S_{S^{(k,s)}}$. It follows that $(t_1, \dots, t_n) = \zeta^B$ for some $B\leq S^{(k,s)}$; hence, $t = \widetilde\zeta^B$. 
Thus, $P\cap\ZZ^{n+1} = \{\widetilde\zeta^B\colon B\leq S^{(k,s)}\textup{ for some } k\}$ and $P\supseteq \mathcal P(S^{(s)})$. We conclude that $P = \mathcal P(S^{(s)})$ and that $\mathcal P(S^{(s)})\cap\ZZ^{n+1} = P\cap\ZZ^{n+1} = \{\widetilde\zeta^B\colon B\leq S^{(k,s)}\}$.
\end{proof}
Let $f^\top_i \colonequals \#F^\top(w)_i$, and write $f^\top = (f^\top_1, \dots, f^\top_n)$. For $f\leq f^\top$, let $D^f(w) = (D^f(w)_1, \dots, D^f(w)_n)$ be the diagram with
\[
D^f(w)_k = (D(w)_k)^{(f_k, i_k)} \qquad\textup{ where } (i_k,k) \in A(w).
\]
\begin{lem}
\label{lem:SBD-explicit}
We have
\[
\{D\colon \textup{there exist $r, F$ so that } (D,r,F) \in\SBD(w)\} = \{D\colon D\leq D^{f}(w) \textup{ for some } f\leq f^\top\}.
\]
\end{lem}
\begin{proof}
Let $(D,r,F)\in\SBD(w)$. Define $(f_1, \dots, f_n)$ by $f_k\colonequals \#F_k = \#D_k - \#D(w)_k$. Because $F\subseteq F^\top(w)$, we know that $f\leq f^\top$. Suppose that $(i_k,k)\in A(w)$ is the $m_k$-th highest square in $D(w)_k$. Writing $d_i$ for the $i$-th highest square in $D_k$, we have $F_k = \{d_{m_k+1} < \dots < d_{m_k+f_k}\}$ with $d_{m+f_k} \leq i_k$. Since $D\in\BD(w)$, Lemma~\ref{lem:BD-characterization} implies that $D_k\setminus F_k \leq D(w)_k$. It follows that $D_k\leq (D^f(w))_k$, and by varying $k$, we deduce that $D\leq D^f(w)$.

Now suppose that $D\leq D^f(w)$. As above, suppose that $(i_k,k) \in A(w)$ is the $m_k$-th highest square in $D(w)_k$, and write $d_i$ for the $i$-th highest square in $D_k$. Let $F_k\colonequals \{d_{m_k+1} < \dots < d_{m_k+f_k}\}$. Then $D\setminus F\leq D(w)$, and furthermore, $d_{m_k+f_k}\leq i_k$. The live square immediately above any square in $F_k$ is $d_{m_k}\in D_k\setminus F_k$, and no two squares in $A(w)$ are linked. Thus, Lemma~\ref{lem:BD-characterization} implies that there exists $r$ so that $(D,r,F)\in\BD(w)$.
\end{proof}
For $D\leq D^f(w)$, let $\widetilde{\wt}(D)\in\ZZ^{n+1}$ denote the vector whose $i$-th coordinate counts the number of squares in the $i$-th row of $D$ for $i \leq n$ and whose $(n+1)$-th coordinate is $\deg(\mathfrak G_w) - |D|$.

Recall that $\widetilde{\mathfrak G}_w$ denotes the \textbf{homogenized Grothendieck polynomial}
\[
\widetilde{\mathfrak G}_w(x_1, \dots, x_n, z) \colonequals \sum_{k=\ell(w)}^{\deg(\mathfrak G_w)}\mathfrak G_w^{(k)}(x_1, \dots, x_n)z^{\deg(\mathfrak G_w) - k}
\]

\newtheorem*{thm:main3}{Theorem~\ref{thm:main3}}
\begin{thm:main3}
Let $w\in S_n$ be a vexillary permutation. Then, the homogenized Grothendieck polynomial $\widetilde{\mathfrak G}_w$ has M-convex support. In particular, each degree component $\mathfrak G_w^{(k)}$ has M-convex support.
\end{thm:main3}

\begin{proof}[Proof of Theorem~\ref{thm:main3}]
By Theorem~\ref{thm:supp-Gw-SBD} and Lemma~\ref{lem:SBD-explicit}, we know that 
\begin{align*}
\supp(\widetilde{\mathfrak G}_w) &= \left\{\widetilde{\wt}(D)\colon D\leq D^f(w) \textup{ for some } f\leq f^\top\right\} \\
&=\sum_{k=1}^n\left\{\widetilde\zeta^{D_k} \colon D_k\leq (D(w)_k)^{(f_k,i_k)}\textup{ for some } f_k \leq f^\top_k\right\}.
\end{align*}
Thus 
\[
\Newton(\widetilde{\mathfrak G}_w) = \sum_{k=1}^n\mathcal P((D(w)_k)^{(i_k)})
\]
is a generalized permutahedron. Furthermore, \cite{schrijver2003}*{Corollary 46.2c} implies that any $t \in \Newton(\widetilde{\mathfrak G_w})\cap\ZZ^{n+1}$ can be written as a sum
\[
t = t_1 + \dots + t_n, \qquad t_i \in \mathcal P((D(w)_k)^{(i_k)}) \cap\ZZ^{n+1}.
\]
Since $t_i = \widetilde\zeta^{D_i}$, we conclude that $t = \widetilde{\wt}(D)\in\supp(\widetilde{\mathfrak G}_w)$ for $D = (D_1, \dots, D_n)$, so $\widetilde{\mathfrak G}_w$ has SNP.

M-convexity of $\supp(\mathfrak G_w^{(k)})$ follows from the equality
\[
\supp(\mathfrak G_w^{(k)}) = \supp(\widetilde{\mathfrak G}_w)\cap\{t\in\RR^{n+1}\colon t_{n+1} = \deg(\mathfrak G_w) - k\}.
\]
\end{proof}

\section{Linear independence of Schubert matroid rank functions}
\label{sec:linear-independence}
We prove Theorem~\ref{thm:main4} and use it to show that our results are sharp.
\begin{defn}
For each $n$, denote by $V_n$ the set of all nonempty subsets of $[n]$ with the following total order: if $I,J\in V_n$, then $I\prec J$ if
\[
\max(I\setminus J)\leq\max(J\setminus I),
\]
where we take $\max(\emptyset)\colonequals 0$.
\end{defn}
\begin{example}
$V_4$ is the chain
\begin{align*}
\{1\}&\prec\{2\}\prec\{1,2\}\prec\{3\}\prec\{1,3\}\prec\{2,3\}\prec\{1,2,3\}\\
&\prec\{4\}\prec\{1,4\}\prec\{2,4\}\prec\{1,2,4\}\prec\{3,4\}\prec\{1,3,4\}\prec\{2,3,4\}\prec\{1,2,3,4\}.
\end{align*}
Note that $V_{n-1}$ is an initial segment of $V_n$.
\end{example}
\begin{defn}
For each $n\geq 1$, define $A_n$ to be the $(2^n-1)\times(2^n-1)$ matrix
\[
A_n = (r_{\SM_n(I)}(J))_{I,J\in V_n}.
\]
\end{defn}
\begin{example}
For $n = 3$ and $n = 4$, we have 
\[A_3=\left(\begin{array}{*{7}c}
1 & 0 & 1 & 0 & 1 & 0 & 1 \\
1 & 1 & 1 & 0 & 1 & 1 & 1 \\
1 & 1 & 2 & 0 & 1 & 1 & 2 \\
1 & 1 & 1 & 1 & 1 & 1 & 1 \\
1 & 1 & 2 & 1 & 2 & 1 & 2 \\
1 & 1 & 2 & 1 & 2 & 2 & 2 \\
1 & 1 & 2 & 1 & 2 & 2 & 3 \\
\end{array}\right),
\quad
A_4=\left(\begin{array}{*{15}c}
1 & 0 & 1 & 0 & 1 & 0 & 1 & 0 & 1 & 0 & 1 & 0 & 1 & 0 & 1 \\
1 & 1 & 1 & 0 & 1 & 1 & 1 & 0 & 1 & 1 & 1 & 0 & 1 & 1 & 1 \\
1 & 1 & 2 & 0 & 1 & 1 & 2 & 0 & 1 & 1 & 2 & 0 & 1 & 1 & 2 \\
1 & 1 & 1 & 1 & 1 & 1 & 1 & 0 & 1 & 1 & 1 & 1 & 1 & 1 & 1 \\
1 & 1 & 2 & 1 & 2 & 1 & 2 & 0 & 1 & 1 & 2 & 1 & 2 & 1 & 2 \\
1 & 1 & 2 & 1 & 2 & 2 & 2 & 0 & 1 & 1 & 2 & 1 & 2 & 2 & 2 \\
1 & 1 & 2 & 1 & 2 & 2 & 3 & 0 & 1 & 1 & 2 & 1 & 2 & 2 & 3 \\
1 & 1 & 1 & 1 & 1 & 1 & 1 & 1 & 1 & 1 & 1 & 1 & 1 & 1 & 1 \\
1 & 1 & 2 & 1 & 2 & 1 & 2 & 1 & 2 & 1 & 2 & 1 & 2 & 1 & 2 \\
1 & 1 & 2 & 1 & 2 & 2 & 2 & 1 & 2 & 2 & 2 & 1 & 2 & 2 & 2 \\
1 & 1 & 2 & 1 & 2 & 2 & 3 & 1 & 2 & 2 & 3 & 1 & 2 & 2 & 3 \\
1 & 1 & 2 & 1 & 2 & 2 & 2 & 1 & 2 & 2 & 2 & 2 & 2 & 2 & 2 \\
1 & 1 & 2 & 1 & 2 & 2 & 3 & 1 & 2 & 2 & 3 & 2 & 3 & 2 & 3 \\
1 & 1 & 2 & 1 & 2 & 2 & 3 & 1 & 2 & 2 & 3 & 2 & 3 & 3 & 3 \\
1 & 1 & 2 & 1 & 2 & 2 & 3 & 1 & 2 & 2 & 3 & 2 & 3 & 3 & 4 \\
\end{array}\right).\]
Because $V_{n-1}$ is an initial segment of $V_n$, the upper left justified $(2^{n-1}-1)\times(2^{n-1}-1)$ submatrix of $A_n$ is equal to $A_{n-1}$.
\end{example}
We would like to show that the columns of $A_n$ are linearly independent. To do this, we will use symmetries of $A_n$ which relate blocks of $A_n$ with $A_{n-1}$. We first give a motivating example.
\begin{example}
Take $A_4$ as above. For each $I\in V_3$, subtract row $I$ from row $I\cup\{n\}$ to get
\begin{center}
\includegraphics[scale=0.9]{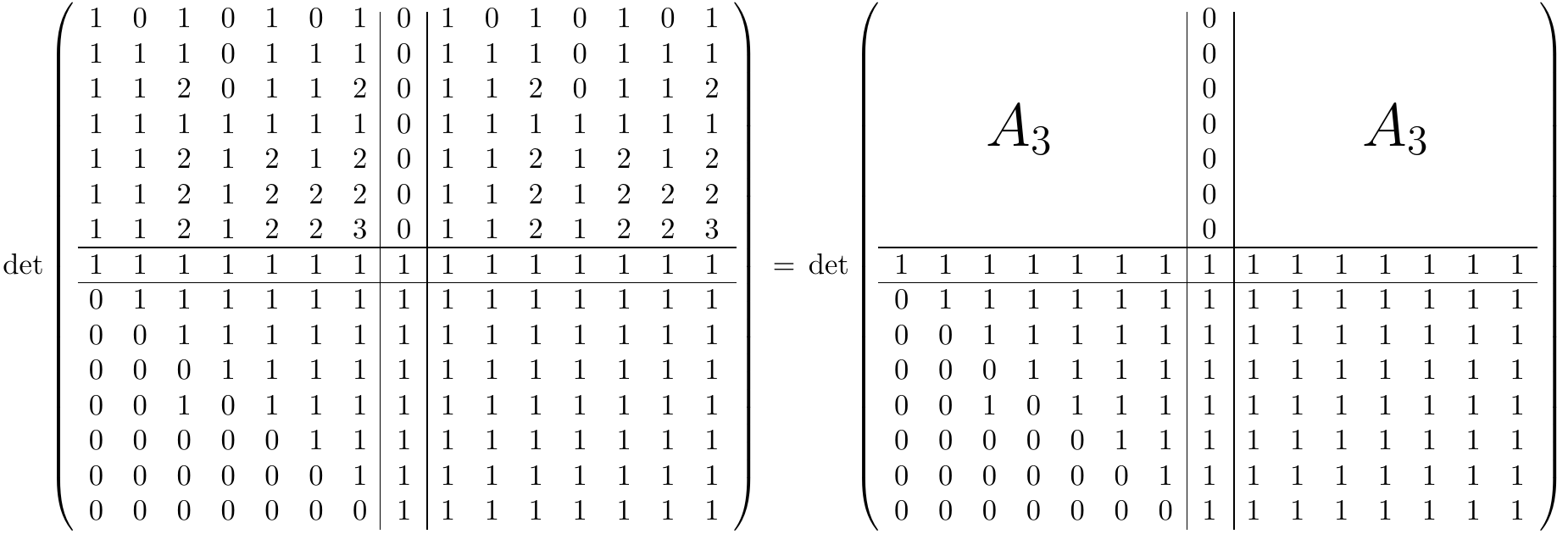}
\end{center}
\end{example}
\begin{lem}
\label{lem:schubert-matroid-rank-1}
The row and column of $A_n$ indexed by $\{n\}$ are given by
\[
r_{\SM_n(\{n\})}(J) = 1 \qquad\textup{ and } \qquad r_{\SM_n(I)}(\{n\}) = \begin{cases} 1 &\textup{ if } n \in I\\ 0 &\textup{ if } n \not\in I\end{cases}
\]
respectively.
\end{lem}
\begin{proof}
It is straightforward to check that
\[
\word_{\{n\}}^n(J) = \begin{cases} ( \dots ( \star &\textup{ if } n \in J\\ ( \dots () &\textup{ if } n\not\in J\end{cases}\quad\textup{ and } \quad \word_I^n(\{n\}) = \begin{cases} ) \dots )\star &\textup{ if } n\in I\\ ) \dots )( &\textup{ if } n\not\in I\end{cases}.
\]
\end{proof}
\begin{lem}
\label{lem:schubert-matroid-rank-2}
Let $I,J\in V_n$. The rank functions $r_{\SM_n(I)}$ satisfy the following properties:
\begin{enumerate}
\item If $n\not\in I$ and $n\not\in J$, then $r_{\SM_n(I)}(J) = r_{\SM_{n-1}(I)}(J)$;
\item If $n\in I$ and $n\not\in J$, then $r_{\SM_n(I)}(J) = r_{\SM_n(I)}(J\setminus\{n\})$;
\item If $n\in I$ and $n\in J$, then $r_{\SM_n(I)}(J) = r_{\SM_n(I\setminus\{n\})}(J\setminus\{n\})+1$;
\item If $n\in I$, $n\not\in J$, and $I\setminus\{n\}\prec J$, then $r_{\SM_n(I)}(J) = r_{\SM_n(I\setminus\{n\})}(J) + 1$;
\item If $I\succ\{n\}$, then $r_{\SM_n(I)}(I\setminus\{n\}) = r_{\SM_n(I\setminus\{n\})}(I\setminus\{n\})$.
\end{enumerate}
\end{lem}
\begin{proof}
If $n\not\in I$ and $n\not\in J$, then $\word_I^n(J) = \word_I^{n-1}(J)$. Thus, $r_{\SM_n(I)}(J) = r_{\SM_{n-1}(I)}(J)$.

If $n\not\in I$ and $n\in J$, then $\word_I^n(J)$ is obtained by appending a $($ to the end of $\word_I^n(J\setminus\{n\})$. Doing so does not change the number of $\star$s or paired $()$s, so $r_{\SM_n(I)}(J) = r_{\SM_n(I)}(J\setminus\{n\})$.

If $n\in I$ and $n\in J$, then $\word_I^n(J)$ is obtained by appending a $\star$ to the end of $\word_{I\setminus\{n\}}^n(J\setminus\{n\})$, so $r_{\SM_n(I)}(J) = r_{\SM_n(I\setminus\{n\})}(J\setminus\{n\})+1$.

If $n\in I$ and $n\not\in J$, then $\word_I^n(J)$ is obtained by appending a $)$ to the end of $\word_{I\setminus\{n\}}^n(J)$. On the other hand, if $I\setminus\{n\}\prec J$, then $\max(I\setminus\{n\}\setminus J) < \max(J\setminus I)$; thus, $\word_{I\setminus\{n\}}^n(J)$ contains an unmatched left parenthesis to the right of all closed parentheses. Combined, we deduce $r_{\SM_n(I)}(J) = r_{\SM_n(I\setminus\{n\})}(J) + 1$.

If $I\succ\{n\}$, then $\word_I^n(I\setminus\{n\})$ is obtained by appending a $)$ to the end of $\word_{I\setminus\{n\}}^n(I\setminus\{n\}) = \star\dots\star$. Thus, $r_{\SM_n(I)}(I\setminus\{n\}) = r_{\SM_n(I\setminus\{n\})}(I\setminus\{n\})$.
\end{proof}

\begin{prop}
\label{prop:linear-independence-of-ranks}
The Schubert matroid rank functions $r_{\SM_n(I)}$ are linearly independent.
\end{prop}
\begin{proof}
We will show that the columns of $A_n$ are linearly independent. First, let $A_n'$ denote the matrix obtained from $A_n$ by subtracting row $I$ from row $I\cup\{n\}$ for each $I \in V_{n-1}$. Lemmas~\ref{lem:schubert-matroid-rank-1} and~\ref{lem:schubert-matroid-rank-2} imply that $A_n'$ has a block decomposition
\begin{center}
\includegraphics{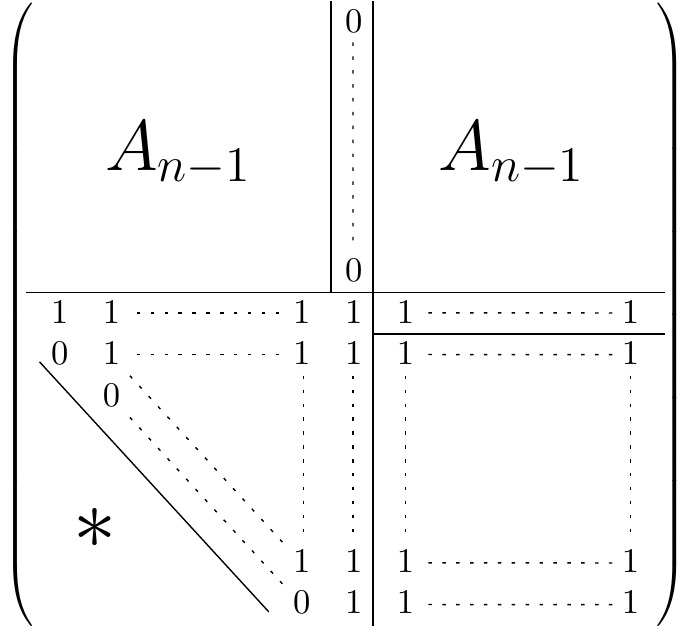}
\end{center}
Let $A_n''$ denote the matrix obtained from $A_n'$ by subtracting each row in $V_n\setminus (V_{n-1}\cup\{n\})$ from the row above it, working from the top row to the bottom row. Then $A_n''$ has a block decomposition
\begin{center}
\includegraphics{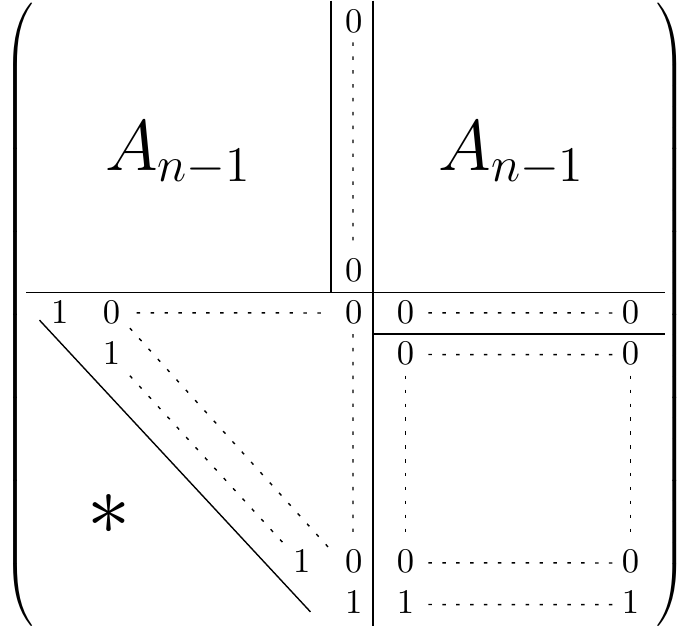}
\end{center}
Let $v_J$ denote the column vector of $A_n''$ indexed by $J\in V_n$, and suppose that 
\begin{equation}
\label{eqn:linear-independence-of-ranks}
\tag{$\diamondsuit$}
\sum_{J\in V_n}c_Jv_J = 0
\end{equation}
is a linear dependence between the vectors $v_J$. The columns of $A_{n-1}$ are linearly independent, so $c_J + c_{J\cup\{n\}} = 0$ for all $J \in V_{n-1}$. Furthermore, comparing coordinates of the dependence~\eqref{eqn:linear-independence-of-ranks} corresponding to $V_n\setminus (V_{n-1}\cup\{n\})$, working from the smallest element to the largest element, gives that $c_J = 0$ for all $J\in V_{n-1}$. It follows that $c_{J\cup\{n\}} = 0$ for all $J\in V_{n-1}$ as well. Thus, the linear dependence~\eqref{eqn:linear-independence-of-ranks} reads $c_{\{n\}}v_{\{n\}} = 0$, and since $v_{\{n\}} \neq 0$, it follows that $c_{\{n\}} = 0$.

We conclude that the columns of $A_n''$, and hence the columns of $A_n$, are independent.
\end{proof}
\newtheorem*{thm:main4}{Theorem~\ref{thm:main4}}
\begin{thm:main4}
Fix $n\geq 1$. The rank functions $r_{\SM_n(I)}$ of Schubert matroids form a basis of the vector space of functions $f\colon 2^{[n]} \to \RR$ satisfying $f(\emptyset) = 0$. In particular,
\begin{itemize}
\item A generalized permutahedron is a Schubitope if and only if its associated submodular function is a $\ZZ_{\geq 0}$-linear combination of rank functions of Schubert matroids, and

\item Two Schubitopes $\mathcal S_D$ and $\mathcal S_{D'}$ are equal if and only if $D$ can be obtained from $D'$ by a permutation of columns.
\end{itemize}
\end{thm:main4}
\begin{proof}[Proof of Theorem~\ref{thm:main4}]
The vector space of functions $f\colon 2^{[n]}\to\RR$ satisfying $f(\emptyset) = 0$ is $(2^n - 1)$-dimensional and contains the $2^n - 1$ functions $r_{\SM_n(I)}$. Proposition~\ref{prop:linear-independence-of-ranks} guarantees that these functions are linearly independent, so they form a basis.

Let $D = (D_1, \dots, D_k)$ be a collection of columns. The submodular function of the Schubitope $\mathcal S_D$ is given by $r_{\SM_n(D_1)} + \dots + r_{\SM_n(D_k)}$; in particular, it is a $\ZZ_{\geq 0}$-linear combination of Schubert matroid rank functions.

Lemma~\ref{lem:gp-from-rank} guarantees that a generalized permutahedron $P$ is uniquely determined by its submodular function $z$. Because Schubitopes are generalized permutahedra, an arbitrary generalized permutahedron $P$ is equal to a Schubitope $\mathcal S_D$ if and only if the submodular function $z$ defining $P$ is a $\ZZ_{\geq 0}$-linear combination of rank functions of Schubert matroid polytopes.

Combined with the linear independence of rank functions of Schubert matroid polytopes, it also follows that two Schubitopes $\mathcal S_D$ and $\mathcal S_{D'}$ are equal if and only if $D$ can be obtained from $D'$ by a permutation of columns.
\end{proof}
\begin{rem}
One can show that $\det(A_n) = 1$, so the Schubert matroid rank functions in fact form a $\ZZ$-basis for the space of functions $f\colon 2^{[n]}\to\ZZ$ with $f(\emptyset) = 0$.
\end{rem}

The following examples provide counterexamples to natural generalizations of Theorems~\ref{thm:main2} and \ref{thm:main3}.
\begin{example}
\label{exp:main2counter}
Consider the non-vexillary permutation $w = 2168534(10)79 \in S_{10}$. We show that the Newton polytope of $\mathfrak G_w^\top$ is not a Schubitope. 
The defining inequalities of $\Newton(\mathfrak G_w^\top)$ show it is a generalized permutahedron. Its submodular function $z$ expands in the basis of Schubert matroid rank functions as
\begin{align*}
	z = r_{\SM_n(\{1\})} - r_{\SM_n(\{2,3,4\})} + 2 r_{\SM_n(\{1,2,3,4\})} + r_{\SM_n(\{3,4,5\})} \\
	\phantom{aaaaaaaa}+r_{\SM_n(\{1,2,3,4,5\})} + r_{\SM_n(\{2,3,4,8\})} + r_{\SM_n(\{1,2,3,4,5,6,7,8\})}.
\end{align*}
Because there is a negative coefficient in this expansion, Theorem~\ref{thm:main4} implies that \linebreak$\Newton(\mathfrak G_w^\top)$ is not a Schubitope.
\end{example}
\begin{example}
\label{exp:main3counter}
Let $w = 14253 \in S_5$. We show that the Newton polytope of $\mathfrak G_w^{(\ell(w) + 1)}$ is not a Schubitope. 
Since $w$ is vexillary, Theorem~\ref{thm:main3} implies $\Newton(\mathfrak G_w^{(\ell(w) + 1)})$ is a generalized permutahedron. Its submodular function $z$ expands in the basis of Schubert matroid rank functions as
\[
	z = r_{\SM_n(\{1,2\})} + r_{\SM_n(\{2,4\})} - r_{\SM_n(\{1,2,4\})} + r_{\SM_n(\{2,3,4\})}.
\]
Because there is a negative coefficient in this expansion, Theorem~\ref{thm:main4} implies that \linebreak$\Newton(\mathfrak G_w^{(\ell(w)+1)})$ is not a Schubitope.
\end{example}

Based on the previous two examples, we conclude with the following conjecture, a generalization of Theorem \ref{thm:main2}.

\newtheorem*{conj:1}{Conjecture \ref{conj:1}}
\begin{conj:1}
	If $w \in S_n$ is vexillary, then $\mathfrak G_w^\top$ is an integer multiple of $\chi_{D^\top(w)}$.
\end{conj:1}
We tested Conjecture \ref{conj:1} for all vexillary $w\in S_n$, $n \leq 9$.


\begin{bibdiv}
\begin{biblist}

\end{biblist}
\end{bibdiv}


\begin{bibdiv}
\begin{biblist}
\bib{fms2018}{article}{
 author={Fink, Alex},
 author={M\'esz\'aros, Karola},
 author={St.\ Dizier, Avery},
 title={Schubert polynomials as integer point transforms of generalized permutahedra},
 journal={Adv.\ Math.},
 volume={332},
 date={2018},
 pages={465--475}
}
\bib{frank2011}{book}{
 author={Frank, Andr\'as},
 title={Connections in Combinatorial Optimization},
 series={Oxford Lecture Series in Mathematics and its Applications},
 volume={38},
 publisher={Oxford University Press},
 date={2011} 
}
\bib{hmms2022}{article}{
	author={Huh, June},
	author={Matherne, Jacob P.},
	author={M\'esz\'aros, Karola},
	author={St.~Dizier, Avery},
	title={Logarithmic concavity of {S}chur and related polynomials},
	journal={Trans. Amer. Math. Soc.},
	volume={375},
	date={2022},
	pages={4411--4427}
}
\bib{ms2020}{article}{
	author={M{\'es}z{\'a}ros, Karola},
	author={St.~Dizier, Avery},
	title={From generalized permutahedra to {G}rothendieck polynomials via flow polytopes},
	journal={Algebr. Comb.},
	volume={3},
	date={2020},
	pages={1197--1229},
	number={5}
}
\bib{ey2017}{article}{
	author={Escobar, Laura},
	author={Yong, Alexander},
	title={Newton polytopes and symmetric Grothendieck polynomials},
	journal={C. R. Math. Acad. Sci. Paris},
	volume={355},
	date={2017},
	pages={831--834},
	number={8}
}
\bib{hafner2022}{article}{
 author={Hafner, Elena S.},
 title={Vexillary Grothendieck Polynomials via Bumpless Pipe Dreams},
 eprint={arXiv:2201.12432},
 date={2022}
}
\bib{ccmm2022}{article}{
 author={Castillo, Federico},
 author={Cid-Ruiz, Yairon},
 author={Mohammadi, Fatemeh},
 author={Monta\~no, Jonathan},
 title={K-polynomials of multiplicity-free varieties},
 date={2022},
 eprint={arXiv:2212.13091}
}
\bib{ls1982}{article}{
 author={Lascoux, Alain},
 author={Sch\"utzenberger, Marcel-Paul},
 title={Structure de Hopf de l'anneau de cohomologie et de l'anneau de Grothendieck d'une vari\'et\'e de drapeaux},
 journal={C.\ R.\ Acad.\ Sc.\ Paris},
 volume={295},
 date={1982},
 number={11},
 pages={629--633}
}
\bib{mty2019}{article}{
 author={Monical, Cara},
 author={Tokcan, Neriman},
 author={Yong, Alexander},
 title={Newton polytopes in algebraic combinatorics},
 journal={Selecta.\ Math.},
 volume={25},
 number={66},
 date={2019},
}
\bib{mssd2022}{article}{
 author={M\'esz\'aros, Karola},
 author={Setiabrata, Linus},
 author={St.\ Dizier, Avery},
 title={On the support of Grothendieck polynomials},
 eprint={arXiv:2201.09452},
 date={2022}
}
\bib{ps2022}{article}{
 author={Pechenik, Oliver},
 author={Satriano, Matthew},
 title={Proof of a conjectured M\"obius inversion formula for Grothendieck polynomials},
 eprint={arXiv:2202.02897},
 date={2022}
}
\bib{psw2021}{article}{
 author={Pechenik, Oliver},
 author={Speyer, David E.},
 author={Weigandt, Anna},
 title={Castelnuovo-Mumford regularity of matrix Schubert varieties},
 eprint={arXiv:2111.10681},
 date={2021}
}
\bib{schrijver2003}{book}{
 author={Schrijver, Alexander},
 title={Combinatorial Optimization. Polyhedra and efficiency.},
 series={Algorithms and Combinatorics},
 volume={24},
 publisher={Springer-Verlag, Berlin},
 date={2003} 
}
\bib{weigandt2021}{article}{
 author={Weigandt, Anna},
 title={Bumpless pipe dreams and alternating sign matrices},
 journal={J.\ Combin.\ Theory, Ser.\ A},
 volume={182},
 date={2021},
 number={105470}
}
\bib{ry2015}{article}{
	author={Ross, Colleen},
	author={Yong, Alexander},
	title={Combinatorial rules for three bases of polynomials},
	journal={S\'{e}m. Lothar. Combin.},
	volume={74},
	date={2015},
	pages={Art. B74a, 11},
}
\bib{ls2021}{article}{
	author={Lam, Thomas},
	author={Lee, Seung Jin},
	author={Shimozono, Mark}
	title={Back stable Schubert calculus},
	journal={Compos. Math.},
	volume={157},
	date={2021},
	pages={883–962}
}

\bib{PanYu23}{article}{
author={Pan, Jianping},
author={Yu, Tianyi},
title={Top-degree components of Grothendieck and Lascoux polynomials},
year={2023},
      eprint={arXiv: 2306.04159}
}

\bib{Yu23}{article}{
author={Yu, Tianyi},
title={Connection between Schubert polynomials and top Lascoux polynomials},
year={2023},
      eprint={arXiv: 2302.03643}
}

\end{biblist}
\end{bibdiv}
\end{document}